\documentclass[a4paper,11pt,reqno]{amsart}
\usepackage{amsmath}
\usepackage[T1]{fontenc}
\usepackage{amssymb}
\usepackage{amsthm,graphicx}
\usepackage{color}
\usepackage[lmargin=2.5 cm,rmargin=2.5 cm,tmargin=3.5cm,bmargin=2.5cm,
paper=a4paper]{geometry}

\newcommand{\Ab}{\mathbf A}

\newcommand{\kp}{\kappa}
\newcommand{\Om}{\Omega}

\newcommand{\Fb}{\mathbf F}

\newcommand{\R}{\mathbb R}
\newcommand{\C}{\mathbb C}
\newcommand{\E}{\mathrm{E}_{\rm gs}(\kappa, H)}

\newcommand{\Es}{E_{\rm surf}}
\newcommand{\fb}{\mathbf{f}}
\newcommand{\omp}{\Omega_+}
\newcommand{\omm}{\Omega_-}
\newcommand{\pomp}{(\partial\Omega)_+}
\newcommand{\pomm}{(\partial\Omega)_-}

\DeclareMathOperator{\curl}{curl}

\newtheorem{thm}{Theorem}[section]
\newtheorem{prop}[thm]{Proposition}
\newtheorem{lem}[thm]{Lemma}

\newtheorem{Ass}[thm]{Assumption}

\newtheorem{proposition}[thm]{Proposition}

\theoremstyle{remark}

\newtheorem{rem}[thm]{Remark}

\numberwithin{equation}{section}

\title[Decay of superconductivity]
{Decay of superconductivity \\ away from the magnetic  zero set}

\author{Bernard Helffer}
\author{Ayman Kachmar}
\address[B. Helffer]{Laboratoire de Math\'ematiques Jean Leray, CNRS,  Universit\'e de Nantes, France, and Laboratoire de Math\'ematiques, Universit\'e de Paris-Sud, Univ Paris-Saclay, France.}
\email{bernard.helffer@univ-nantes.fr}
\address[A. Kachmar]{Department of Mathematics, Lebanese University, Hadat, Lebanon.}
\email{ayman.kashmar@gmail.com}
\date{\today}
\begin{document}

\begin{abstract}
We establish exponential bounds on the Ginzburg-Landau order parameter away from the curve where the applied magnetic field vanishes. In the units used in this paper, the estimates are valid when the parameter measuring the strength of the applied magnetic field is comparable with the Ginzburg-Landau parameter.  This completes a previous work by the authors analyzing the case when  this strength was much higher.  Our results display the distribution of surface and bulk superconductivity and are valid under the assumption that the magnetic field is H\"older continuous.
\end{abstract}

\maketitle 

\section{Introduction}\label{hc2-sec:int}

\subsection{The functional}
In non-dimensional units, the Ginzburg-Landau functional is defined
as follows,
\begin{equation}\label{eq-3D-GLf}
\mathcal E(\psi,\Ab)=\int_\Omega \left(|(\nabla-i\kappa H\Ab)\psi|^2-\kappa^2|\psi|^2+\frac{\kappa^2}2|\psi|^4+(\kappa H)^2|\curl\Ab-B_0|^2\right)\,dx\,,
\end{equation}
where:
\begin{itemize}
\item $\Omega\subset\R^2$ is an open, bounded and simply connected
set with a $ C^\infty$  boundary\,; 
\item $(\psi,\Ab)\in H^1(\Omega;\mathbb C)\times H^1(\Omega;\mathbb
R^2)$\,;
\item $\kappa>0$ and $H>0$ are two parameters\,;
\item $B_0$ is a  real-valued function in $L^2(\Omega)\,$. 
\end{itemize}
The superconducting sample is supposed to occupy a long cylinder with vertical axis and  horizontal cross section 
$\Omega$. The parameter $\kappa$ is the Ginzburg-Landau parameter  that expresses the properties of the superconducting material. The
applied magnetic field is $\kappa HB_0 \vec{e}$, where $\vec{
e}=(0,0,1)$.
 The configuration pair $(\psi,\Ab)$ 
describes the state of superconductivity as follows:
$|\psi|^2$ measures the density of the superconducting Cooper
pairs, $\curl\Ab$ measures the induced magnetic field in the
sample and $j:=(i\psi,\nabla\psi -i\kappa H\Ab\psi)$ measures the induced super-current. Here $(\cdot,\cdot)$ denotes the inner product in $\C$ defined as follows,  $(u,v)=u_1v_1+u_2v_2$ where $u=u_1+iu_2$ and $v=v_1+iv_2\,$.

At equilibrium, the state of the superconductor is described by the (minimizing) configurations $(\psi,\Ab)$ that realize the following ground state energy
\begin{equation}\label{eq-gse}
\E=\inf\{\mathcal E(\psi,\Ab)~:~(\psi,\Ab)\in H^1(\Omega;\mathbb
C)\times H^1(\Omega;\mathbb R^2)\}\,.
\end{equation}
Such configurations are critical points of the functional introduced in \eqref{eq-3D-GLf},  that is they solve the following system of Euler-Lagrange equations ($\nu$ is the unit inward normal on the boundary)
\begin{equation}\label{eq:GL}
\left\{
\begin{array}{rll}
-\big(\nabla-i\kappa H\Ab\big)^2\psi&=\kp^2(1-|\psi|^2)\psi &{\rm in}\ \Om\,,\\
-\nabla^{\perp}  \big(\curl\Ab-B_0\big)&= (\kp H)^{-1}{\rm Im}\big(\overline{\psi}\,(\nabla-i\kp H {\bf A})\psi\big) & {\rm in}\ \Om\,,\\
\nu\cdot(\nabla-i\kappa H\Ab)\psi&=0 & {\rm on}\ \partial \Om \,,\\
{\rm curl}{\bf A}&=B_0 & {\rm on}\ \partial \Om \,.
\end{array}
\right.
\end{equation}
Once a  choice of  $(\kappa,H)$ is fixed, the notation   $(\psi,\Ab)_{\kappa,H}$  stands for a solution of \eqref{eq:GL}.  When $B_0$ belongs to $C^0(\overline{\Omega})$, we  introduce two constants $\beta_0$ and $\beta_1$ that will  play a central role in this paper:
\begin{equation}\label{eq:ep0}
\beta_0:=\sup_{x\in\overline{\Omega}}|B_0(x)|\,\quad{\rm and}\quad \beta_1:=\sup_{x\in\partial\Omega}|B_0(x)|\,.
\end{equation}

\subsection{The case with a constant magnetic}

A huge mathematical literature is devoted to the analysis of the functional in \eqref{eq-3D-GLf} when the magnetic field is constant. This corresponds to taking $B_0=1$ in \eqref{eq-3D-GLf}. The two monographs
\cite{FH-b, SaSe} and the references therein are mainly devoted to this subject. One important situation is  the transition from {\it bulk} to {\it surface} superconductivity. This happens when the parameter $H$ increases between two critical values $H_{C_2}$ and $H_{C_3}$ called the second and third critical fields respectively.

In this analysis the {\it de\,Gennes constant} plays a central role. This constant is universal and  defined as follows
\begin{equation}\label{eq:Theta0}
\Theta_0=\inf_{\xi\in\R}\Big\{\inf_{\|u\|_2=1}\Big(\int_0^\infty\big(|u'(t)|^2+(t-\xi)^2|u(t)|^2\big)\,dt\Big)\Big\}\,.
\end{equation}
Furthermore, it is known  (cf. \cite{FH-b}) that
\begin{equation}\label{eq:Theta_0a}
\frac12<\Theta_0<1\,.
\end{equation}

The de\,Gennes constant   appears indeed in the asymptotics of $H_{C_3}$ for $\kappa$ large
$$
H_{C_3}\sim \Theta_0^{-1}\kappa\,,
$$ 
while we have for the second critical field
 $$H_{C_2}\sim \kappa \,.
$$
To be more specific, if $b>0$ is a constant and  $(\psi,\Ab)_{\kappa, H}$ is a minimizer of the functional in \eqref{eq-3D-GLf} for $H=b\kappa$ (and $B_0=1$), the concentration of $\psi$ in the limit $\kappa\to\infty$ depends strongly on  $b\,$.  

If $0<b<1$, then  $\psi$ is uniformly distributed in the domain $\Omega$ (cf. \cite{Kac, SS02})\,.
If $1<b<\Theta_0^{-1}$, then $\psi$ is concentrated on the surface and decays exponentially in the bulk (cf. \cite{CR, Pa02})\,.
 If $b>\Theta_0^{-1}$, then $\psi=0$ (cf. \cite{HePa, LuPa1}).
The critical cases when  $b$ is close to $1$   or $\Theta_0^{-1}$ are thoroughly analyzed in \cite{FK-am, FoHe04}.

\subsection{The case with a non-vanishing magnetic field}

The case of a non-constant magnetic field $B_0$ satisfying the assumptions
$$B_0\in C^0(\overline\Omega)\quad{\rm and}\quad \inf_{x\in\overline\Omega}B_0(x)>0\,,$$ 
is qualitatively similar to the constant magnetic field case.  This situation is reviewed in \cite[Sec.~2.2]{HK}.
Surface superconductivity is studied in \cite{FH-b}, while the transition to the normal solution is discussed in \cite{Ra}.

\subsection{The case with a vanishing magnetic field}

The results in this paper are valid for a large class of applied magnetic fields, see Assumption~\ref{ass:mf*} below. However, one interesting  situation covered by our results is  the case where the applied magnetic field has a non-trivial zero set. In the presence of such an applied magnetic field, we will study  the concentration  of the minimizers $(\psi,\Ab)_{\kappa,H}$ of \eqref{eq-3D-GLf} in the asymptotic limit $\kappa\to+\infty$ and $H\approx\kappa\,$. Unlike the results in \cite{FH-b, Ra} that only investigate surface superconductivity, the situation  discussed here includes bulk superconductivity as well.

The discussion in this subsection is focusing on magnetic fields that  satisfy:
\begin{Ass}\label{ass:mf}{\bf[On the applied magnetic field]}~
\begin{enumerate}
\item The function $B_0$ is in  $C^{1}(\overline{\Omega})$\,.
\item The set $\Gamma:=\{x\in\overline\Omega~:~B_0(x)=0\}$ is non-empty and consists of a finite disjoint  union of  simple smooth curves.
\item $\Gamma\cap\partial\Omega$ is either {\bf empty} or a {\bf finite} set.
\item For all  $x\in\overline\Omega\,$, $|B_0(x)|+|\nabla B_0(x)|\neq 0\,$.
\item The set $\Gamma$  is allowed to intersect $\partial\Omega$  transversely. More precisely, if $\Gamma\cap\partial\Omega\not=\emptyset\,$, then on this set,
$\nu\times\nabla B_0\not=0\,$, where $\nu$ is the normal vector field along $\partial\Omega\,$. 
\end{enumerate}
\end{Ass}
A much weaker assumption will be described later (cf. Assumption~\ref{ass:mf*}).
Under Assumption~\ref{ass:mf}, we may introduce the following two non-empty open sets
\begin{equation}\label{eq:Om-pm}
\Omega_+=\{x\in\Omega~:~B_0(x)>0\}\quad{\rm and}\quad \Omega_-=\{x\in\Omega~:~B_0(x)<0\}\,.
\end{equation}
The boundaries of $\Omega_\pm$ are given as follows
$$\partial\Omega_\pm=
\Gamma\cup(\overline{\Omega}_\pm\cap\partial\Omega)\,.
$$
%
Magnetic fields satisfying  Assumption~\ref{ass:mf}   are discussed  in many contexts:  
\begin{itemize}
\item  In geometry, this appears in \cite{M} under the appealing question: {\it can we hear the zero locus of a magnetic field\,?}
\item In the semi-classical analysis of the spectrum of Schr\"odinger operators with magnetic fields satisfying Assumption~\ref{ass:mf} (and $\Gamma\subset\Omega$). These operators are extensively studied in \cite{DR, HeMo-JFA, HKo}.
\item In the study of the time-dependent Ginzburg-Landau equations \cite{AlHe, AHP},    applied magnetic fields as in  Assumption~\ref{ass:mf} naturally appear in  the presence of applied electric currents.
\item For superconducting surfaces submitted to constant magnetic fields \cite{CL},  the constant magnetic field may induce a smooth sign-changing magnetic field on the surface.
\item In the transition from normal to superconducting configurations \cite{PK}, one meets the problem of determining $H$ such that the ground state energy in \eqref{eq-gse} vanishes on a curve meeting transversally the boundary. The results in \cite{PK} are  sharpened in \cite{Att3, Miq}.
\item The asymptotics of  the ground state energy in \eqref{eq-gse} and the concentration of the corresponding minimizers for large values of $\kappa$  and $H$ is analyzed in  \cite{Att, Att2, HK, HK2}. 
\end{itemize}
Of particular importance to us are the results of K. Attar in \cite{Att}.  These results hold under Assumption~\ref{ass:mf},  for $H=b\kappa$ with  $b>0$ constant. One of the results in \cite{Att} is that the ground state energy in \eqref{eq-gse} satisfies, as $\kappa\to +\infty\,$,
\begin{equation}\label{eq:Att}
\E=\kappa^2\int_\Omega g(b|B_0(x)|)\,dx+o(\kappa^2)\,.
\end{equation}
Here the function  $g(\cdot)$, which was  introduced  by Sandier-Serfaty in \cite{SS02}, is a continuous non-decreasing  function defined on $[0,\infty)$ and vanishes  on $[1,\infty)$ (cf.   \eqref{eq:g} for more details).

 K. Attar also obtained an interesting formula displaying the local distribution of the minimizing order parameter $\psi$.  If $(\psi,\Ab)_{\kappa,H}$ is a minimizer of the functional in \eqref{eq-3D-GLf} for $$H=b\kappa\,,$$
  and if  $\mathcal D$ is an open set in $\Omega$ with a smooth boundary, then, as $\kappa\to +\infty\,$,
\begin{equation}\label{eq:Attbis}
\int_{\mathcal D}|\psi (x)|^4\,dx=-2\int_{\mathcal D} g(b|B_0(x)|)\,dx+o(1)\,.
\end{equation}
The interest for an $L^4$ control of the order parameter comes back to Y. Almog (see \cite{Al} and the discussion in the book  \cite[Ch.~12, Sec.~12.6]{FH-b}).

The formula in \eqref{eq:Attbis} shows that  $\psi$ is weakly localized in the neighborhood of  $\Gamma$,  $\mathcal V \left(\frac 1 b\right)$, where:
\begin{equation}\label{eq:nu}
\mathcal V \left(\epsilon\right):=\Big\{ x \in \Omega\,,\, |B_0(x)| \leq \epsilon \Big\}\,.
\end{equation}

 For taking account of the boundary effects  (the surface superconductivity should play a role like in the constant magnetic field case) we also introduce  in $\partial \Omega$ the subset
\begin{equation}\label{eq:nu-bnd}
 \mathcal V^{\rm bnd}\left(\epsilon \right):= \big\{ x\in \partial \Omega\,,\,\Theta_0 |B_0(x)| \leq \epsilon \big\}\,.
 \end{equation}

 We would like to measure the strength of the (exponential) decay of the minimizing order parameter $\psi$ in 
the  domains 
\begin{equation}\label{eq:om(l)} 
\omega\left(\frac 1b\right):= \Omega \setminus\mathcal  V\left(\frac 1b\right)
\,.
\end{equation}

Note the role played by the  two  constants introduced in \eqref{eq:ep0}. 
If $\frac1b\geq \beta_0\,$, then $\mathcal V(\frac 1b)=\Omega\,$.  For this reason we will focus on the values of $b$ above $\beta_0^{-1}$. We also observe that, if $\frac1b\geq \Theta_0\beta_1\,$, then $\mathcal V^{\rm bnd}(\frac1 b)=\partial\Omega\,$. Hence, boundary effects are expected to appear when $b<\frac1{\Theta_0\beta_1}\,$.\\

\begin{figure}\label{fig1}
\begin{center}
\includegraphics[scale=0.4]{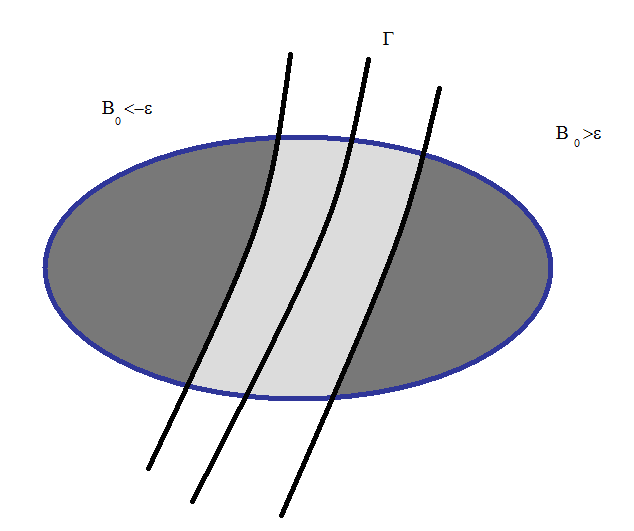} 
\end{center}
\caption{Illustration of Regime~I for $H=b\kappa$ and $b=1/\varepsilon$\,: Superconductivity is destroyed in the dark regions and survived on the entire boundary.}
\end{figure}
\begin{figure}\label{fig2}
\begin{center}
\includegraphics[scale=0.5]{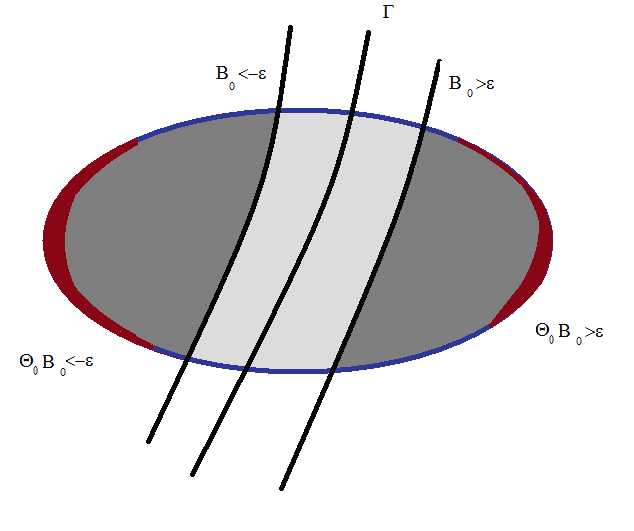} 
\end{center}
\caption{Illustration of Regime~II for $H=b\kappa$ and $b=1/\varepsilon$\,: Superconductivity is  also destroyed on the  boundary parts $\{\Theta_0|B_0(x)|>\varepsilon\}\cap\partial\Omega\,$.}
\end{figure}

\begin{figure}\label{fig3}
\begin{center}
\includegraphics[scale=0.8]{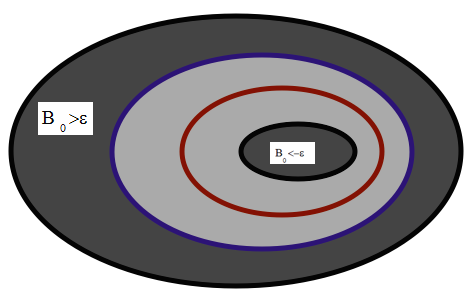} 
\end{center}
\caption{Illustration of Regime~II when $\{B_0=0\}\cap\partial\Omega=\emptyset\,$, $H=b\kappa\,$, $b=1/\varepsilon$ and $\varepsilon$ is small: Superconductivity is   destroyed on the entire  boundary and is concentrated in the set $\{|B_0|<\varepsilon\}$.}
\end{figure}

Loosely speaking, we would like to prove that, for all values of  $b\geq \beta_0^{-1}$,   the density $|\psi|^2$ is exponentially small (in the $L^2$-sense) outside the set $\mathcal V(\frac 1b)\cup\mathcal V^{\rm bnd}(\frac 1b)$. This will lead us to two distinct regimes:\\

{\bf Regime~I:} For  $\beta_0^{-1}< b\leq (\Theta_0\beta_1)^{-1}$, $\mathcal V^{\rm bnd}(\frac 1b)=\partial\Omega$ and $\partial\Omega$ carries surface superconductivity everywhere. This is illustrated in Figure~1.\\

{\bf Regime~II:} For  $b>(\Theta_0\beta_1)^{-1}$, we will get that $\psi$ is exponentially small outside the set $\mathcal V^{\rm bnd}(\frac 1b)$. Here we have two cases:
\begin{itemize}
\item As $b$ increases,  surface superconductivity shrinks to the points of $\{x\in\partial\Omega,~B_0(x)=0\}$, provided that this set is non-empty (cf. Figure~2). 
\item If $\{B_0(x)=0\}\cap\partial\Omega=\emptyset\,$, then, for sufficiently large values of $b$, no surface superconductivity is left (cf. Figure~3). 
\end{itemize}

Regime~II   is consistent with the results of   \cite[Thm. 3.6]{HK} devoted to the complementary regime where $b\gg 1$ as $\kappa \rightarrow +\infty\,$.\\
The results in this paper confirm the   behavior described  in these two regimes and  are valid  under a much weaker  assumption than Assumption~\ref{ass:mf} (cf. Assumption~\ref{ass:mf*} below).

The transition to the normal state is studied in \cite{Att3, Miq, PK}. This happens, when $\kappa$ is large, for $H\sim c_*\kappa^2$ (equivalently $b \sim  c_*\kappa$), where $c_*>0$ is a constant explicitly defined by the domain $\Omega$ and the function $B_0\,$.\\

\subsection{Main results}~

In this paper, we will  first work under the following assumption:
\begin{Ass}\label{ass:mf*}~
\begin{itemize}
\item The function $B_0$ is in $C^{0,\alpha}(\overline\Omega)$ for some $\alpha\in(0,1)$\,;
\item The constants $\beta_0$ and $\beta_1$ in \eqref{eq:ep0} satisfy $\beta_1\geq \beta_0>0\,$.
\end{itemize}
\end{Ass}
Note that this assumption is much weaker than Assumption \ref{ass:mf}.
With the previous notation our main theorem is:
\begin{thm}\label{thm:exp-dec*}
{\bf[Exponential decay outside the  superconductivity region]}\\
Suppose that Assumption~\ref{ass:mf*} holds, that $b>\beta_0^{-1}$  and let  $O$ 
be an open set such that
$\overline{O}\subset \omega\big(\frac1b\big)$\,, 
 where $\omega(\frac 1b)$ is the domain introduced in \eqref{eq:om(l)}

There exist $\kappa_0>0$,  $C>0$ and $ \alpha_0>0$ 
such that, if
$\kappa\geq\kappa_0$ and $(\psi,\Ab)_{\kappa,H}$ is a solution of \eqref{eq:GL} for $H=b\kappa$\,,
then the following  inequality holds
\begin{equation}\label{eq:exp-dec*} 
\| \psi \|_{H^1(O)}\leq C \, e^{-\alpha_0\kappa}\,.
\end{equation}
Furthermore, if $b>(\Theta_0\beta_1)^{-1}$, then the estimate in \eqref{eq:exp-dec*}  holds when the  open set $O$ satisfies 
$$
\overline{O}\subset \Big\{x\in\partial\Omega,~\Theta_0|B_0(x)|>\frac1b\Big\}\cup{\omega}\left(\frac 1b\right) \,.
$$
\end{thm}

The proof of Theorem~\ref{thm:exp-dec*} follows from the stronger conclusion of   Theorem~\ref{thm:exp-dec}, establishing Agmon like estimates.

\begin{rem}\label{rem:B0pm}{\bf[Sign-changing magnetic fields]}~\\
In addition to Assumption~\ref{ass:mf*}, suppose that $\omp$ and $\omm$ 
are {\bf non-empty}. The constant $\beta_0$ in \eqref{eq:ep0} can be expressed as follows
$$\beta_0=\max(\beta_0^+,\beta_0^-)\quad{\rm where~}\beta_0^\pm=\sup_{x\in\overline{\Omega_\pm}}|B_0(x)|\,.$$
We will discuss the conclusion of Theorem~\ref{thm:exp-dec*} when $\beta_0^+<\beta_0^-\,$.  We have:
\begin{itemize}
\item If $ (\beta_0)^{-1}<b<(\beta_0^+)^{-1}$, then $\omega(\frac1b)\cap \Omega_+=\emptyset\,$. Consequently, the exponential decay occurs in $\omega(\frac1b)\cap\omm\,$.
\item If $(\beta_0^+)^{-1}\leq b\,$, then the exponential decay occurs in both $  \omega(\frac1b)\cap\omp$ and $ \omega(\frac1b)\cap\omm\,$.
\end{itemize}
The situation when $\beta_0^-<\beta_0^+$ can be discussed similarly. 
Next, we suppose that the two sets
$$\pomp:=\{x\in\partial\Omega~,~B_0(x)>0\}\quad{\rm and} \quad\pomm:=\{x\in\partial\Omega~,~B_0(x)<0\}$$ 
are non-empty, and we express the constant $\beta_1$ in \eqref{eq:ep0} as follows
$$\beta_1=\max(\beta_1^+,\beta_1^-)\quad{\rm where~}\beta_1^\pm=\sup_{x\in\overline{(\partial\Omega)_\pm}}|B_0(x)|\,.$$
According to Theorem~\ref{thm:exp-dec*}, when $\beta_1^+<\beta_1^-$ and $(\beta_1)^{-1}<b<(\beta_1^+)^{-1}\,$, then the exponential decay occurs on $\{x\in\partial\Omega,~\Theta_0b|B_0(x)|>1\}\cap\pomm\,$, since 
$\{x\in\partial\Omega,~\Theta_0 b|B_0(x)|>1\}\cap \pomp=\emptyset\,$. 
\end{rem}
Our next result discusses the optimality of   Theorem~\ref{thm:exp-dec*}. This theorem determines a part of the boundary  where  the order parameter (the first component $\psi$ of the minimizer) is  exponentially small. Outside this part of the boundary, we will prove that the $L^4$ norm of the order parameter  is not exponentially  small. In physical terms,  superconductivity is present there.   

The statement of Theorem~\ref{thm:surface-sc} involves the following notation:
\begin{itemize}
\item For all $t>0$, $\widetilde\Omega(t)=\{x\in\R^2~:~{\rm dist}(x,\partial\Omega)<t\}\,$.
\item By smoothness of $\partial\Omega\,$, there exists a geometric constant $t_0$ such that, 
for all $x\in\widetilde\Omega(t_0)\,$, we may assign a unique point
 $p(x)\in\partial\Omega$  such that ${\rm dist}(p(x),x)= {\rm dist}(x,\partial\Omega)$.
\item If $b>0$, we define the open subset in $\R^2$
\begin{equation}\label{eq:1.15*}
\widetilde\Omega(t_0,b)=\{x\in\widetilde\Omega(t_0)~:~ 1< b|B_0(p(x))|<\Theta_0^{-1}\}\,.
\end{equation}
\item $E_{\rm surf}:[1,\Theta_0^{-1})\to (-\infty,0)$  is the surface energy function which will be defined in \eqref{eq:Pa02} later. This function is continuous and non-decreasing.
\item  If $\widetilde\Omega(t_0,b)\not=\emptyset\,$, we define the following distribution in $\mathcal D'\big(\widetilde\Omega(t_0,b)\big)$:
\begin{equation}
C_c^\infty\big(\widetilde\Omega(t_0,b)\big)\ni\varphi\mapsto\mathcal T_b(\varphi)=-2\int_{\widetilde\Omega(t_0,b)\cap\partial\Omega}\sqrt{\frac1{b|B_0(x)|}}\Es\big(b|B_0(x)|\big)\,\varphi(x)\,ds(x)\,,
\end{equation}
where $ds$ is the surface measure on $\partial\Omega$.
\item If $D\subset\overline\Omega\,$, we introduce the local Ginzburg-Landau energy in $D$ as follows
\begin{equation} \label{eq:1.15n }
\mathcal E(\psi,\Ab;D)=\displaystyle\int_{D}
\Big(|(\nabla-i\kappa H\Ab)\psi|^2-\kappa^2|\psi|^2+\frac{\kappa^2}2 |\psi(x)|^4\Big)\,dx\,.
\end{equation}
\item $\mathbf 1_\Omega$ denotes the characteristic function of the set $\Omega\,$.
\end{itemize}

\begin{thm}\label{thm:surface-sc}{\bf [Existence of surface superconductivity]}~

Suppose that Assumption~\ref{ass:mf*} holds, that $b>\beta_0^{-1}$ and that $\widetilde\Omega(t_0,b)\not=\emptyset$, where $\beta_0$ is the constant introduced in
\eqref{eq:ep0}. 
If $(\psi,\Ab)_{\kappa,H}$ is a minimizer of the functional in \eqref{eq-3D-GLf} for $H=b\kappa\,$, then as $\kappa\to\infty\,$,  we have the following weak convergence\footnote{ A distribution $T_\kappa$ converges weakly to a distribution $T$ in $\mathcal D'(U)$ if, for all $\varphi\in C_c^\infty(U)$, $T_\kappa(\varphi)\to T(\varphi)$.} 
\begin{equation}\label{eq:1.13*}
\kappa\mathbf 1_\Omega |\psi_{\kappa,H}|^4\rightharpoonup \mathcal T_b\quad{\rm in~}\mathcal D'\big(\widetilde\Omega(t_0,b)\big)\,.
\end{equation}
\end{thm}

\begin{rem}\label{rem:lb-psi}
Theorem~\ref{thm:surface-sc} demonstrates the existence of surface superconductivity.
 We can interpret the assumption in Theorem~\ref{thm:surface-sc} in two different ways.
\begin{itemize}
\item If $H=b\kappa$, $b>0$  is fixed and $x_0\in\partial\Omega$, then to find superconductivity near $x_0$, this point should satisfy $1< b|B_0(x_0)|<\Theta_0^{-1}$.
\item If $x_0\in\partial\Omega$ is fixed and $|B_0(x_0)|$ is small, then to find superconductivity near $x_0$, the intensity of the applied magnetic field should be increased in such a manner that $H=b\kappa$ and $1< b|B_0(x_0)|<\Theta_0^{-1}$.
\end{itemize}
\end{rem}

Our last result confirms that the region $\{B_0(x)<\frac\kappa H\}$ carries superconductivity everywhere. To state it, we will use the following notation:
\begin{itemize}
\item If $p,q\in\partial\Omega$, ${\rm dist}_{\partial\Omega}(p,q)$ denotes the (arc-length) distance in $\partial\Omega$ between $p$ and $q$\,.
\item For $x_0\in\R^2$ and $r>0\,$,   we denote by $Q_r(x_0) = x_0 + (-r/2,r/2)^2$  the interior of the square of center $x_0$ and side $r$.  When $x_0=0$, we write $Q_r=Q_r(0)$.
\item 
For $(x,\ell)\in\overline\Omega\times(0,t_0/2)$, we will use the following notation:
\begin{equation}\label{defW}
\mathcal W(x_0,\ell)=
\left\{\begin{array}{cl}
\{x\in\overline\Omega~:~{\rm dist}_{\partial\Omega}(p(x),x_0)<\ell~\mbox{ and }~{\rm dist}(x,\partial\Omega)<2\ell\}&{\rm ~if~}x_0\in\partial\Omega\,,\\
Q_{2\ell}(x_0)&{\rm ~if~}x_0\in\Omega\,.
\end{array}\right.
\end{equation}
\end{itemize}

\begin{thm}\label{thm:blk-sc*}{\bf[The superconductivity region]}~

Suppose that Assumption~\ref{ass:mf*} holds for some $\alpha\in(0,1)$, $b>0$  and
$\frac{2}{2+\alpha}<\rho<1$ be two constants. Let $x_0\in\overline\Omega$ such that $|B_0(x_0)|<\frac1b$.

There exist $\kappa_0>0$,  a function $ {\rm r}:[\kappa_0,+\infty)\to\R_+$ 
such that $\lim_{\kappa\to +\infty}{\rm r}(\kappa)=0$ and, 
for all $\kappa\geq\kappa_0$ and for all critical point  $(\psi,\Ab)_{\kappa,H}$ of the functional in \eqref{eq-3D-GLf} with $H=b\kappa$\,, 
the following  two inequalities hold,
$$
\left|\frac1{|\mathcal W(x_0,\kappa^{-\rho})|}\int_{\mathcal W_{x_0}(\kappa^{-\rho})}|\psi(x)|^4\,dx+2g\big(b|B_0(x_0)|\big)\right|\leq {\rm r}(\kappa)
$$
and
$$
\Big|\mathcal E\Big(\psi,\Ab;\mathcal W(x_0,\kappa^{-\rho})\Big)-\kappa^2g\big(b|B_0(x_0)|\big)\Big|\leq \kappa^2{\rm r}(\kappa)\,.
$$
Here $ g(\cdot)$ is the continuous function appearing in \eqref{eq:Att} and \eqref{eq:Attbis}
(see  Subsection \ref{ss2.1} for its definition and properties).
\end{thm}

The result in Theorem~\ref{thm:blk-sc*} is a variant of  the formula in \eqref{eq:Attbis} valid  for applied magnetic fields which are only H\"older continuous, thereby  generalizing the results by Attar \cite{Att} and Sandier-Serfaty \cite{SS02}. This will be clarified further in Remark~\ref{rem:blk-sc'}.

\begin{rem}\label{rem:blk-sca}
Let us choose fixed constants $\gamma$ and $\rho$ such that $\frac{2}{2+\alpha}<\rho<1$ and $0<\gamma<1-\rho$.
Our proof of Theorem~\ref{thm:blk-sc*} yields that  the constant $\kappa_0$ and the  function $r(\kappa)$ in Theorem~\ref{thm:blk-sc*} can be selected
independently of the point $x_0$ provided that
\begin{itemize}
\item $\kappa^{-2\gamma} \leq  b|B_0(x_0)|<1$\,;
\item $x_0\in\partial\Omega$ or ${\rm dist}(x_0,\partial\Omega)\geq 4\kappa^{-\rho}$\,.
\end{itemize} 
The condition ${\rm dist}(x_0,\partial\Omega)\geq 4\kappa^{-\rho}$ ensures that $Q_{2\kappa^{-\rho}}(x_0)\subset\overline{\Omega}$, which is needed in the proof of Theorem~\ref{thm:blk-sc*}.
\end{rem}

\begin{rem}\label{rem:blk-sc'}
Let $\gamma \in (0, \frac{\alpha}{2+ \alpha})$. If we assume furthermore the following geometric condition
\begin{equation}\label{eq:cond-B0=0}
\big|\{x\in\overline\Omega~,~|B_0(x)|\leq \kappa^{-2\gamma}\} \big|=o(1)\quad(\kappa\to\infty)\,,
\end{equation}
then  Theorem~\ref{thm:blk-sc*} implies the weak convergence
$$ |\psi_{\kappa,H} (\cdot) |^4\rightharpoonup  -2g\big(b|B_0(\cdot )|\big)\quad{\rm in~}\mathcal D'(\Omega)\,.$$ 
In \eqref{eq:cond-B0=0}, we have used the following notation. If $E\subset\R^2$, $|E|$ denotes the Lebesgue (area) measure of $E$.
Note that the condition in \eqref{eq:cond-B0=0} holds under Assumption~\ref{ass:mf} considered in \cite{Att}.
\end{rem}


The rest of the paper is organized as follows. In Section~\ref{sec:pre}, we collect various results that will be used throughout the paper. Section~\ref{sec:exp-dec} is devoted to the proof of Theorem~\ref{thm:exp-dec*}. In Section~\ref{sec:surf}, we present the proof of Theorem~\ref{thm:surface-sc}. Finally, we prove Theorem~\ref{thm:blk-sc*} in Section~\ref{sec:blk}. 

 In the proofs, we avoid the use of the {\it a priori} elliptic $L^\infty$-estimates, whose derivation is quite complicated (cf. \cite[Ch.~11]{FH-b}), thereby providing  new proofs for the results in \cite{Pa02, SS02}. To our knowledge, 
these $L^\infty$-estimates have not been established 
when the magnetic field $B_0$ is only H\"older continuous.

\section{Preliminaries}\label{sec:pre}

\subsection{The bulk energy function}\label{ss2.1}

The energy function $g(\cdot)\,$, hereafter called the bulk energy, has been  constructed  in \cite{SS02}. We will recall its construction here. It plays a central role in the study of `bulk' superconductivity, both for two and three dimensional problems (cf. \cite{FKP-jmpa, FK-cpde}). Furthermore, it is related to the periodic solutions of \eqref{eq:GL} and the  Abrikosov energy  (cf. \cite{AS, FK-am}).

For $b\in\,(0,+\infty)\,$, $r>0\,$, and
$Q_r=\,(-r/2,r/2)\,\times\,(-r/2,r/2)$\,, we define the functional,
\begin{equation}\label{eq:rGL}
F_{b,Q_r}(u)=\int_{Q_r}\left(b|(\nabla-i\Ab_0)u|^2-|u|^2+\frac{1}2|u|^4\right)\,dx\,,
\quad \mbox{ for } u\in H^1(Q_r)\,.
\end{equation}
Here, $\Ab_0$ is the magnetic potential,
\begin{equation}\label{eq:A0}
\Ab_0(x)=\frac12(-x_2,x_1)\,,\quad \mbox{ for } x=(x_1,x_2)\in \R^2\,.
\end{equation}
We define the  Dirichlet and Neumann ground state energies by
\begin{align}
&e_D(b,r)=\inf\{F_{b,Q_r}(u)~:~u\in H^1_0(Q_r)\}\,,\label{eq:eD}\\
&e_N(b,r)=\inf\{F_{b,Q_r}(u)~:~u\in H^1(Q_r)\}\,.\label{eq:eN}
\end{align}
We can define $g(\cdot)$ as follows (cf. \cite{Att, FK-cpde, SS02})
\begin{equation}\label{eq:g}
\forall~b>0\,,\quad g(b)=\lim_{r\to\infty}\frac{e_D(b,r)}{|Q_r|}=\lim_{r\to\infty}\frac{e_N(b,r)}{|Q_r|}\,,
\end{equation}
where  $|Q_r| =r^2$ denotes the area of $Q_r\,$.\\
Furthermore, there exists a universal constant $C>0$ such that
\begin{equation}\label{eq:g(b)*}
\forall~b>0\,,\quad\forall~r>1\,,\quad g(b)\leq \frac{e_D(b,r)}{|Q_r|}\leq \frac{e_N(b,r)}{|Q_r|}+\frac{C}r\leq g(b)+\frac{2C}r\,.\end{equation}
One can show that  the
function $g(\cdot)$ is a  non decreasing continuous function such
that
\begin{equation} \label{propg}
g(0)=-\frac12 \,,\quad  g(b) < 0  \mbox{ when } b <1\,,\, \mbox{ and } \quad g(b)=0 \mbox{ when } b\geq 1\,.
\end{equation}

\subsection{The magnetic Laplacian}

We need two  results about the magnetic Laplacian.
 The first result concerns the Dirichlet magnetic  Laplace operator in a bounded set $\Omega$  with a strong constant magnetic field $B$, that is
$$-(\nabla-iB\Ab_0)^2\quad{\rm in~}L^2(\Omega)\,,$$
with the  Dirichlet condition
$$ u=0~{\rm on~}\partial\Omega\,.$$
Here $\Ab_0$ is the vector field introduced in \eqref{eq:A0}, with  $\curl\Ab_0=1$. It is based on  the elementary spectral inequality:
\begin{lem}\label{lem:ml*}
For all $B\in\R$ and $\phi\in H^1_0(\Omega)$, it holds
$$\int_\Omega|(\nabla-iB\Ab_0)\phi |^2\,dx\geq |B|\int_\Omega|\phi(x)|^2\,dx\,.$$
\end{lem}
The second result concerns the Neumann magnetic  Laplace operator in a bounded set $\Omega$  with a strong constant magnetic field $B$, that is
$$-(\nabla-iB\Ab_0)^2\quad{\rm in~}L^2(\Omega)\,,$$
with the (magnetic) Neumann condition
$$\nu\cdot(\nabla-iB\Ab_0)u=0~{\rm on~}\partial\Omega\,.$$
Here $\nu$ is the unit inward normal vector on $\partial\Omega$. The 
asymptotic behavior of the groundstate energy as $|B|\to\infty$ is well known  (cf. \cite{HM, LuPajmp} and \cite[Prop.~8.2.2]{FH-b}):
\begin{lem}\label{lem:ml}
There exist $\hat \beta_0>0$ and $C>0$ such that, if $|B|\geq \hat \beta_0$ and $\phi\in H^1(\Omega)\,$, 
$$\int_\Omega|(\nabla-iB\Ab_0)\phi|^2\,dx\geq \left(\Theta_0|B|-C|B|^{3/4}\right)\int_\Omega|\phi|^2\,dx\,.$$
\end{lem}

\subsection{Universal bound on the order parameter}

If $(\psi,\Ab)$ is a solution of \eqref{eq:GL}, then $\psi$ satisfies in $\Omega$ (cf. \cite[Prop.~10.3.1]{FH-b})
\begin{equation}\label{eq:psi<1}
|\psi(x) |\leq 1 \,.
\end{equation}

\subsection{The magnetic energy}

Let us introduce the space of vector fields
\begin{equation}\label{eq:H1div}
H^1_{\rm div}(\Omega)=\{\Ab\in H^1(\Omega;\R^2)~:~{\rm div}\Ab=0{\rm ~in~}\Omega\quad{\rm and}\quad\nu\cdot\Ab=0~{\rm on~}\partial\Omega\}\,.
\end{equation}
The functional in \eqref{eq-3D-GLf} is invariant under the gauge transformations $(\psi,\Ab)\mapsto (e^{i\phi}\psi,\Ab+\nabla\phi)$. Consequently, if $(\psi,\Ab)$ solves \eqref{eq:GL}, we may apply a gauge transformation such that the new configuration $(\widetilde\psi = e^{i\phi}\psi, \widetilde A=A+\nabla\phi)$ is a solution of \eqref{eq:GL} and furthermore $\widetilde\Ab\in H^1_{\rm div}(\Omega)$.
Having this in hand, we always assume that every critical/minimizing configuration $(\psi,\Ab)$ satisfies
$\Ab\in H^1_{\rm div}(\Omega)$ which simply amounts to a gauge transformation.

For given $B_0 \in L^2(\Omega)$, there exists a unique vector field satisfying
\begin{equation}\label{eq:F}
\Fb\in H^1_{\rm div}(\Omega)
\quad{\rm and}\quad\curl\Fb=B_0\,.
\end{equation}
Actually,  $\Fb=\nabla^\bot f$ where $f\in H^2(\Omega)\cap H^1_0(\Omega)$ is the unique solution of  $-\Delta f=B_0$  . 

\begin{rem}
By the elliptic Schauder H\"older estimates (see for example Appendix E.3 in \cite{FH-b}),  if in addition $B_0 \in C^{0,\alpha}(\overline{\Omega})$ for some $\alpha >0\,$, then the vector field $\Fb$ is smooth of class $C^{1,\alpha}(\overline{\Omega})\,$.\\
\end{rem}

We recall the following result from \cite{Att}:

\begin{proposition}\label{lem:Att}
Let $\gamma\in(0,1)$ and  $0<c_1<c_2$ be fixed constants. Suppose that $B_0\in L^2(\Omega)$. 
There exist $\kappa_0>0$ and $C>0$ such that, if $\kappa\geq\kappa_0$, $c_1\,\kappa\leq H \leq c_2\,\kappa$ and if $(\psi,\Ab)_{\kappa,H}\in H^1(\Omega)\times H^1_{\rm div}(\Omega)$ is a minimizer of \eqref{eq-gse}, then
$$\|\Ab-\Fb\|_{C^{0,\gamma}(\overline{\Omega})}\leq \frac{C}{\kappa}\,.$$
\end{proposition}

The proof of Proposition~\ref{lem:Att} given in \cite{Att} is made under the assumption $B_0\in C^\infty(\overline\Omega)$, but it still holds under the weaker assumption $B_0\in L^2(\Omega)$.

The next result gives the existence of  a useful gauge transformation that allows us to approximate the vector field $\Fb$ by a vector field generating a constant magnetic field. It is similar to the  result in  \cite[Lem.~A.3]{Att}, but the difference here is that we only  assume $\Fb\in C^{1,\alpha}(\overline {\Omega})$ instead of $C^2$.

\begin{lem}\label{lem:Gauge-Att}
Let $\alpha\in(0,1)$, $r_0>0$ and $B_0\in C^{0,\alpha}(\overline\Omega)$. There exists $C>0$  and for any $a\in \overline{\Omega}$ a  function  $\varphi_a \in C^{2,\alpha} ( \mathbb R^2)$ such that, if
$r\in(0,r_0]$ and $\ B(a,r)\cap\Omega\not=\emptyset\,,$ then
$$\forall~x\in\overline{B(a,r)\cap\Omega}\,,\quad 
|\Fb(x)-B_0(a)\Ab_0(x-a)-\nabla\varphi_a(x)|\leq C\, r^{1+\alpha}\,.$$
Here $\Fb$ is the vector field satisfying \eqref{eq:F}.
\end{lem}
\begin{proof}[Proof of Lemma~\ref{lem:Gauge-Att}]
Since the boundary of $\Omega$ is smooth and $ \Fb\in C^{1,\alpha}(\overline{\Omega};\R^2)$, the  vector field $\Fb$ admits an extension $\widehat\Fb$  in $C^{1,\alpha}(\R^2;\R^2)$.  We get in this way an extension $\widehat B_0 = \curl \widehat \Fb $ of $B_0$ in $C^{0,\alpha} (\mathbb R^2)$.   We now define in $\mathbb R^2$, the two vector fields
$$\widetilde \Fb(y)=\widehat \Fb(a+y)\,,\quad \mathbf {\widetilde A}(y)=\left(\int_0^1 s\widehat B_0(a+sy)\,ds\right)(-y_2,y_1)\,.$$

Clearly,
$\curl\widetilde\Fb=\curl \mathbf{\widetilde  A}=\widehat B_0(a+y)$.
Consequently, by integrating the closed $1$-form associated with $\widetilde F -\mathbf {\widetilde A}$,  there exists a function $ \widetilde\varphi\in C^{2,\alpha}(\mathbb R^2)$   such that $$\widetilde\Fb-\nabla\widetilde\varphi=\mathbf {\widetilde A}\,,~\, \widetilde \phi (0)=0\,.$$
Since $\widehat B_0\in C^{0,\alpha}(\mathbb R^2)$, 
$\mathbf {\widetilde A} (y)=B_0(a)(-y_2,y_1)+\mathcal O(r^{1+\alpha})$ in  $\overline{B(0,r)}$. 
We then define the function $\varphi_a$ by  $\varphi_a(x)=\widetilde\varphi(x-a)+B_0(a)\Big(a_2x_1-a_1x_2\Big)$. This implies
$$\forall~x\in\overline{B(a,r)}\,,\quad 
|{\bf \widehat F} (x)-B_0(a)\Ab_0(x-a)-\nabla\varphi_a(x)|\leq C\, r^{1+\alpha}\,,$$
and the lemma by restriction to $\overline{\Omega}$.
\end{proof}

\subsection{Lower bound of the kinetic energy term}
The main result in this subsection is:
\begin{prop}\label{lem:qf}  
Let $0<c_1<c_2$ 
 be fixed constants. Suppose that $\alpha\in(0,1]$ and $B_0\in C^{0,\alpha}(\overline\Omega)$. 
 There exist  $\kappa_0>0$ and $C>0$ such that the following is true, with
\begin{equation}
\sigma(\alpha) = \frac{2\alpha}{3+\alpha}\,.
\end{equation}
\begin{enumerate}
\item For 
\begin{itemize}
\item $\kappa\geq\kappa_0$, $c_1\, \kappa\leq H\leq c_2\, \kappa$\,;
\item $(\psi,\Ab)_{\kappa,H}$  a solution of \eqref{eq:GL}\,;
\item $\phi\in H^1(\Omega)$ satisfies ${\rm supp\,}\phi\subset\{x\in\overline\Omega,~|B_0(x)|>0\}$\,,
\end{itemize}
we have
$$\int_\Omega|(\nabla -i\kappa H\Ab)\phi \, (x)|^2\,dx\geq
\Theta_0\kappa H\int_\Omega
\big(|B_0(x)|-C \kappa^{-\sigma(\alpha)}\big) |\phi(x)|^2\,dx\,.$$
\item If in addition $\phi=0$ on $\partial\Omega$, then
$$\int_\Omega|(\nabla -i\kappa H\Ab)\phi\, (x)|^2\,dx\geq
\kappa H\int_\Omega \big(|B_0(x)|-C\kappa^{-\sigma (\alpha)}\big)|\phi(x)|^2\,dx\,.$$
\end{enumerate}
\end{prop}
The estimates in Items~(1) and (2)   in this proposition are known   when the vector field $\Ab$ is $C^2$, independent of $(\kappa,H)$,    $\curl\Ab\not=0$ and $B_0$ is replaced by $\curl \Ab$ (cf. Lemma~\ref{lem:ml} and \cite{HeMo-JFA}).

For $\alpha=1$ (i.e. $B_0$ is Lipschitz) the errors in Proposition~\ref{lem:qf} and Lemma~\ref{lem:ml} are of the same order.

~

\begin{proof}[Proof of Proposition~\ref{lem:qf}]~
Let us choose an arbitrary $\phi\in H^1(\Omega)$. All constants below are independent of $\phi$. For the sake of simplicity, we will work under the additional assumption that ${\rm supp\,}\,\phi\subset\{B_0>0\}$.\\

{\bf Step~1.} {\it Decomposition of the energy via a partition of unity.}~

For $\ell >0$
 we 
consider the partition of unity in $\R^2$
$$\sum_{j}\chi_j^2=1\,,\quad
\sum_{j}|\nabla\chi_j|^2\leq C\, \ell^{-2}\quad{\rm in~}\R^2\,,\quad{\rm and}\quad {\rm supp\,}\chi_j\subset B(a_j^\ell,\ell)\,.$$
Here the construction is first done for $\ell =1$ and then for general $\ell >0$ by dilation. Hence the constant $C$ is independent of $\ell$. Although  the points $(a_j^\ell)$ depend on $\ell$, we omit below the reference to 
$\ell$ and write $a_j$ for $a_j^\ell$.

In what follows, we will use this partition of unity with 
$$\ell=\kappa^{-\rho}\,,\quad \quad 0<\rho<1\quad {\rm and}\quad\kappa{\rm~ large~ enough}.$$

Using this partition of unity, we may estimate from below  the kinetic energy term as follows
\begin{equation}\label{eq:ke-0}
\int_\Omega|(\nabla-i\kappa H\Ab)\phi|^2\,dx\geq
\sum_{j} \left( \int_\Omega|(\nabla-i\kappa H\Ab)(\chi_j\phi)|^2\,dx-C\ell^{-2}\int_\Omega|\chi_j\phi|^2\,dx\right) \,.
\end{equation}
Let $\alpha_j(x)=(x-a_j)\cdot\big(\Ab(a_j)-\Fb(a_j)\big)$, where $\Fb$ is the vector field in \eqref{eq:F}. Note the useful decomposition $$\Ab(x)-\nabla\alpha_j=\Fb(x)+\big(\Ab(x)-\Fb(x)\big)
-\big(\Ab(a_j)-\Fb(a_j)\big)\,.
$$
 By Proposition~\ref{lem:Att}, we have in  $B(a_j,\ell)\cap\Omega$,
\begin{equation}\label{eq:ke-1}
\begin{aligned}
|(\nabla -i\kappa H\Ab)(\chi_j\phi) |^2&=|(\nabla -i\kappa H(\Ab-\nabla\alpha_j))(e^{-i\kappa H\alpha_j}\chi_j\phi)|^2\\
&\geq (1-\ell^{\delta})|(\nabla-i\kappa H\Fb)e^{-i\kappa H\alpha_j}\chi_j\phi|^2
-\ell^{-\delta}\kappa^2H^2 \ell^{2\gamma}\,\|\Ab-\Fb\|_{C^{0,\gamma}(\overline{\Omega})}^2|\chi_j\phi|^2\\
&\geq (1-\ell^{\delta})|(\nabla-i\kappa H\Fb)(e^{-i\kappa H\alpha_j}\chi_j\phi)|^2
-CH^2\ell^{(2\gamma -\delta)} |\chi_j\phi|^2\,.
\end{aligned}
\end{equation}
Here $\delta>0$ and $\gamma\in(0,1)$ are two  parameters to be chosen  later.

By Lemma~\ref{lem:Gauge-Att}, we may define a smooth function $\varphi_j$  in $B(a_j,\ell)\cap\Omega$ such that,
$$
| \Fb(x)-\nabla\varphi_j(x)-|B_0(a_j)|\Ab_0(x-a_j)|\leq C \, \ell^{1+\alpha}\,,$$
 where $C>0$ is independent of $j$.
 
Consequently, there exists $C>0$ such that, for all $j$,
\begin{multline}\label{eq:ke-2}
|(\nabla-i\kappa H\Fb)(e^{-i\kappa H\alpha_j}\chi_j\phi) |^2\geq (1-\ell^{\delta})
|(\nabla-i\kappa H|B_0(a_j)|\Ab_0(x-a_j))e^{-i\kappa H\varphi_j}e^{-i\kappa H\alpha_j}\chi_j\phi|^2\\-C\kappa^2H^2\ell^{2+2\alpha-\delta}|\chi_j\phi|^2\,.
\end{multline}

{\bf Step~2.} {\it The case ${\rm supp\,}\phi \,\subset\{x\in\overline\Omega,~B_0(x)>0\}$ and $\phi=0$ on $\partial\Omega\,$.}\\

The assumption on the support of $\phi$ yields that $\chi_j\phi\in H^1_0(\Omega)$. Collecting \eqref{eq:ke-1}, \eqref{eq:ke-2} and the spectral inequality in Lemma~\ref{lem:ml*}, we get the existence of $C>0$ such that for all $j$ 
\begin{multline*}
\int_\Omega|(\nabla-i\kappa H\Ab)(\chi_j\phi)|^2\,dx
\geq (1-2\ell^{\delta})\kappa H\int_\Omega|B_0(a_j)|\,|\chi_j\phi|^2\,dx\\-
CH^2(\ell^{2\gamma-\delta} +\kappa^2\ell^{2+2\alpha-\delta})\int_\Omega|\chi_j\phi|^2\,dx\,.
\end{multline*}
Since $B_0$ is in $C^{0,\alpha}(\overline\Omega)$, we have $B_0(x)=B_0(a_j)+\mathcal O(\ell^\alpha)$ in $B(a_j,\ell)\,$. Thus
\begin{multline*}
\int_\Omega|(\nabla-i\kappa H\Ab)(\chi_j\phi)|^2\,dx
\geq \kappa H\int_\Omega|B_0(x)|\,|\chi_j\phi(x)|^2\,dx\\-
CH^2  (\ell^\alpha+ \ell^\delta + \ell^{2\gamma-\delta} +\kappa^2\ell^{2+2\alpha-\delta})\int_\Omega|\chi_j\phi (x)|^2\,dx\,.
\end{multline*}
After summation and using that $\sum_j\chi_j^2=1\,$, we get
\begin{multline*}
\int_\Omega|(\nabla-i\kappa H\Ab)\phi|^2\,dx\\
\geq \kappa H  \left( \int_\Omega |B_0(x)|\,|\phi(x)|^2\,dx  -C(\ell^\alpha + \ell^\delta + \ell^{2\gamma-\delta} +\kappa^2\ell^{2+2\alpha-\delta} + \kappa^{-2} \ell^{-2}) \int_\Omega  |\phi|^2\,dx\right).
\end{multline*}
Hence the goal is to choose, when $\kappa \rightarrow +\infty\,$ and with $\ell=\kappa^{-\rho}$\,, the parameters  $\rho\,$, $\delta\,$, $\gamma$ and $\alpha$ in
 order to minimize  the sum 
\begin{equation}\label{sum}
\Sigma_0 (\kappa,\ell)
:=\ell^\alpha + \ell^\delta + \ell^{2\gamma-\delta} +\kappa^2\ell^{2+2\alpha-\delta} + \kappa^{-2} \ell^{-2}.
\end{equation}
If we take $\delta =\gamma $, which corresponds to give the same {\rm order} for the second and  the third terms in $\Sigma_0$\,,
we obtain with $\ell=\kappa^{-\rho}$
$$\int_\Omega|(\nabla-i\kappa H\Ab)\phi|^2\,dx
\geq \kappa H \int_\Omega\Big(|B_0(x)|-C(\kappa^{-\rho\alpha}+\kappa^{-\rho\gamma}+\kappa^{2-(2+2\alpha-\gamma)\rho}+\kappa^{2\rho-2}\Big)|\phi(x)|^2\,dx.$$
In the remainder, to minimize the error for the two last terms, we select $\rho$ such that $$2-(2+2\alpha-\gamma)\rho =2\rho-2\,,$$
 i.e. $$\rho=4/ (4+2\alpha -\gamma)\,.
 $$
 Getting the condition $0<\rho<1$ satisfied leads to  the condition $\alpha>\gamma/2\,$. We select $\gamma=\frac23\alpha\,$.
 This  choice is optimal since
\begin{equation*}
\sigma(\alpha):=\max_{0<\gamma<2\alpha}\sigma_0(\alpha,\gamma)=\sigma_0\left(\alpha,\frac{2\alpha}3\right)=\frac{2\alpha}{3+\alpha}\,,
\end{equation*}
 where
 $$
\sigma_0(\alpha,\gamma)=\min\left(\frac{4\alpha}{4+2\alpha-\gamma},\frac{4\gamma}{4+2\alpha-\gamma},\frac{2(2\alpha-\gamma)}{4+2\alpha-\gamma}\right)\,.
$$

This finishes the proof of Item (2) in Proposition~\ref{lem:qf}\,.\\

{\bf Step~3.} {\it The case ${\rm supp\,}\,\phi\subset\{x\in\overline\Omega\,,~B_0(x)>0\}$.}\\
We continue with the choice $\delta=\gamma=\frac23\alpha$ and $\rho=4/(4+2\alpha-\gamma)$.
We collect the inequalities in \eqref{eq:ke-1}, \eqref{eq:ke-2} and  Lemma~\ref{lem:ml} and  write
\begin{multline*}
\int_\Omega|(\nabla-i\kappa H\Ab)(\chi_j\phi) |^2\,dx
\geq (1-2\ell^{2\alpha/3})\, \kappa H \int_\Omega\Big(\Theta_0|B_0(a_j)|-C(\kappa H)^{-1/4}\Big)\,|\chi_j\phi|^2\,dx\\-
CH^2 \kappa^{-\sigma(\alpha)} \int_\Omega|\chi_j\phi|^2\,dx\,.
\end{multline*}
Since $B_0\in C^{0,\alpha}(\overline\Omega)$, we can replace $B_0(a_j)$ by $B_0(x)$ on the support of $\chi_j$ modulo an error $\mathcal O(\ell^\alpha)$. We insert the resulting estimate into \eqref{eq:ke-0}
and use that $\sum_j\chi_j^2=1$ to get,
$$\int_\Omega|(\nabla-i\kappa H\Ab)\phi|^2\,dx
\geq \kappa H \int_\Omega\Big(\Theta_0|B_0(x)|-C(\kappa^{-\sigma(\alpha)}+\kappa^{-1/2})\Big)|\phi(x)|^2\,dx\,. $$
Observing that $\sigma(\alpha) \leq \frac 12\,$, we have achieved  the proof of Item (1) in Proposition~\ref{lem:qf}\,.
\end{proof}

\section{Exponential decay}\label{sec:exp-dec}
\subsection{Main statements}
We recall the definition of the de\,Gennes constant $\Theta_0$ in \eqref{eq:Theta0}, and the two constants $\beta_0,\beta_1$ in \eqref{eq:ep0}. For all $\lambda\in(0,\beta_0)$, we introduce the two functions  on $\omega(\lambda)$:
\begin{equation}\label{eq:dist-om}
t_\lambda(x)={\rm dist}\big(x,\partial{\omega}(\lambda)\big)\quad{\rm and}\quad
\zeta_\lambda(x)={\rm dist}\big(x,\Omega\cap\partial\omega(\lambda)\big)\,,
\end{equation}
where $\omega(\cdot)$ is the domain introduced in \eqref{eq:om(l)}.

\begin{thm}\label{thm:exp-dec}{\bf[Exponential decay outside the  superconductivity region]}\\
Let $c_1$ and $c_2$ be two constants such that   $\beta_0^{-1}<c_1<c_2\,$. Suppose that 
Assumption~\ref{ass:mf*} holds for some $\alpha\in(0,1)$. There exists $\mu_0>0$ and for all $\mu\in(0,\mu_0)$,  there exist $\kappa_0>0\,$, $C>0$ and $\hat \alpha>0$ such that, if
$$\kappa\geq\kappa_0,\quad c_1\kappa\leq H\leq 
c_2\kappa \,,$$ and $(\psi,\Ab)_{\kappa,H}$ is a solution of \eqref{eq:GL}\,,
then the following  inequalities hold:
\begin{enumerate}
\item {\bf Decay in the interior:}
$$
\int_{\omega(\lambda)\cap\{t_\lambda(x)\geq\frac1{\sqrt{\kappa H}}\}}
\Big(|\psi (x)|^2+{\frac{1}{\kp H}}|(\nabla-i\kappa H\Ab)\psi\,(x)|^2\Big)\,\exp\Big(2\hat \alpha\sqrt{\kappa H}\,t_\lambda(x)\Big)dx
\leq
\frac{C}{\kappa}\,,
$$
where $\lambda=\displaystyle\frac\kappa H+\mu$\, ;
\item {\bf Decay up to the boundary:}
$$
\int_{\omega(\beta)\cap\{\zeta_{\beta}(x)\geq \frac{1}{\sqrt{\kappa H}}\}}
\Big(|\psi(x)|^2+{\frac{1}{\kp H}}|(\nabla-i\kappa H\Ab)\psi\, (x) |^2\Big)\,\exp\Big(2\hat \alpha\sqrt{\kappa H}\,\zeta_{\beta}(x)\Big)dx
\leq
\frac{C}\kappa\,, 
$$
where  $\beta=\Theta_0^{-1}\left(\displaystyle\frac\kappa H+\mu\right)\,$.
\end{enumerate}
\end{thm}
\begin{rem}\label{rem:exp-dec}
Theorem~\ref{thm:exp-dec} says that, for  $\mu>0$ sufficiently small,
 bulk superconductivity breaks down in the region  $\{x\in\Omega,~|B_0(x)|\geq \frac\kappa H+\mu\}$
 and  that surface superconductivity breaks down in the region $\{x\in\partial\Omega,~\Theta_0|B_0(x)|\geq \frac\kappa H+\mu\}$\,. This is illustrated in Figures~1~and~2\,.
\end{rem}

\begin{rem}\label{rem:PanCMP}
 In the constant magnetic field case, $B_0=1\,$, Theorem~\ref{thm:exp-dec} is proved by Pan~\cite{Pa02}, in response to  a conjecture by Rubinstein \cite[p.~182]{R}. Our proof of Theorem~\ref{thm:exp-dec} is simpler than the one in \cite{Pa02} since we  do not use the {\it a priori} elliptic $L^\infty$-estimates, whose derivation is not easy (cf. \cite[Ch.~11]{FH-b}).
\end{rem}

\begin{rem}\label{rem:BonFo}
On a technical level, one can still avoid to use  the $L^\infty$-elliptic estimates in the proof of Theorem~\ref{thm:exp-dec} when the magnetic field is {\bf constant}, by establishing {\it a weak} decay estimate on the order parameter (namely $\|\psi\|_2=\mathcal O(\kappa^{-1/4})$). This has been done by Bonnaillie-No\"el and Fournais in \cite{BonF} and then generalized by Fournais-Helffer to non-vanishing continuous magnetic fields in \cite[Cor.~12.3.2]{FH-b}. However, in the sign-changing field case and the regime considered  in Theorem~\ref{thm:exp-dec}, the weak decay estimate as in \cite{BonF} does not hold. 

The substitute of the weak decay estimate in our proof  is the use of a (local) gauge transformation. This has been used earlier to estimate the Ginzburg-Landau energy  (cf. \cite{K-JFA, Att3}),  and the exponential decay of the order parameter for non-smooth magnetic fields (cf. \cite{AK}).   We will extend this method for obtaining  local estimates in Theorems~\ref{thm:ub-l4} and \ref{thm:lb-psi}.
\end{rem}

\begin{rem}
The conclusion in Theorem~\ref{thm:exp-dec*} is a simple consequence of Theorem~\ref{thm:exp-dec} and the estimate in Proposition~\ref{lem:Att}.
Actually, if $ O$ is an open set independent of $\kappa$ such that $  \overline{O}\subset  \omega(\kappa/H)$, then 
$$  O\subset { \omega} \left(\frac{\kappa}H+\mu\right) $$
for $\mu$ sufficiently small, and
$${\rm dist}\Big(x,\partial\omega\left(\frac{\kappa}H+\mu\right)\Big)\geq c_\mu \quad{\rm in}~O\,,$$
for a constant $c_\mu>0\,$.

Similarly, when $O$ is an open set independent of $\kappa$ and 
$$ \overline{O} \subset { \omega}(\kappa/H)\cup\{x\in\partial\Omega,~\Theta_0|B_0(x)|<\kappa/H\}\,,$$ then 
$$ O\subset   \omega\left(\Theta_0^{-1}\Big(\frac{\kappa}H+\mu\Big)\right) $$
for $\mu$ sufficiently small, and
$$ {\rm dist}\bigg(x,\partial\omega\left(\Theta_0^{-1}\Big(\frac{\kappa}H+\mu\Big)\right)\bigg)\geq \hat c_\mu \quad{\rm in}~O\,,$$
for a constant $\hat c_\mu>0\,$.
\end{rem}

The rest of this section is devoted to the proof of Theorem~\ref{thm:exp-dec}, which follows the scheme of the  proof  of the semi-classical Agmon estimates  (cf. \cite[Ch.~12]{FH-b} and references therein).

Suppose  that the parameters $\kappa$ and $H$ have the same order, i.e. 
$$\kappa\geq\kappa_0\quad{\rm and}\quad 
c_1\kappa\leq H\leq c_2\kappa\,,$$
where $\kappa_0\geq 1$ is supposed sufficiently large (this condition will appear in the proof below). 
Suppose also  that $$c_2 > c_1> \beta_0^{-1}\,,$$ where $c_1, c_2$ are fixed constants and $\beta_0$ was introduced in \eqref{eq:ep0}.

\subsection{Useful inequalities}

For all $\gamma>0\,$, we extend  to $\overline\Omega$  the definitions of $t_\gamma$ and $\zeta_\gamma$ given in \eqref{eq:dist-om} as follows
\begin{equation}\label{eq:t(x)}
t_\gamma(x)=
\left\{
\begin{array}{ll}
{\rm dist}\big(x,\partial\omega(\gamma)\big)&{\rm if~}x\in\omega(\gamma)\\
0&{\rm if~}x\in\overline\Omega\setminus \omega(\gamma))
\end{array}\right.
\end{equation}
 and
\begin{equation}\label{eq:zeta(x)} 
\zeta_\gamma(x)=
\left\{
\begin{array}{ll}
{\rm dist}\big(x, \Omega\cap\partial \omega(\gamma)\big)&{\rm if~}x\in\omega(\gamma)\\
0&{\rm if~}x\in\overline\Omega\setminus\omega(\gamma)
\end{array}\right.\,.\end{equation}
 In the sequel, we will add conditions on $\gamma$ to ensure that $\omega(\gamma)\not=\emptyset\,$.

Let $\tilde{\chi} \in C^\infty(\R)$ be a non negative function satisfying
$$
\tilde{\chi}=0\ {\rm on}\ (-\infty,\frac12]\,,\quad \tilde{\chi}=1\ {\rm on}\ [1,\infty)\,.
$$
Define the   functions $\chi_\gamma $, $\eta_\gamma $, $f_\gamma $ and $g_\gamma$ on $\Om$ as follows:
$$\chi_\gamma (x)=\tilde{\chi}\big(\sqrt{\kp H} t_\gamma (x)\big)\,,\quad \eta_\gamma (x)=\tilde{\chi}\big(\sqrt{\kp H} \zeta_\gamma(x)\big)\,,$$
\begin{equation}\label{eq:f,g}
f_\gamma (x)=\chi_\gamma(x) \exp \big(\hat \alpha \sqrt{\kp H}\,t_\gamma (x)\big)\quad{\rm and}\quad
g_\gamma (x)=\eta_\gamma(x) \exp \big(\hat \alpha \sqrt{\kp H}\,\zeta_\gamma (x)\big)\,,\end{equation}
where $\hat \alpha$ is a positive number whose value will be fixed later.

Let $h\in\{f_\gamma ,g_\gamma\}$. We multiply both sides of the first equation in   \eqref{eq:GL} by $h^2\overline\psi$ and then   integrate by parts over $\omega(\gamma)$. We get
\begin{align}
\int_{\omega(\gamma)} \Big( \big|(\nabla-i\kp H {\bf
A})(h\psi) \big|^2-\kappa^2h^2|\psi|^2-|\nabla h|^2|\psi|^2 \Big)\,dx\leq 
0\,. \label{exp1*}                                                  
\end{align}
In the computations below, the constant $C$ is independent of $\hat \alpha,\gamma,\kappa$ and $H$. We estimate the term involving $\nabla h$ as follows
$$
\int_{\omega(\gamma)}|\nabla h|^2|\psi|^2\,dx \leq  2 \hat \alpha^2\kp H\,  \|h \psi\|^2_{L^2(\omega(\lambda))}+ C\, \kp H\,  T(h)\,,
$$
where
\begin{equation}\label{eq:T(h)}
T(h):=
\begin{cases}
\displaystyle\int_{\omega(\gamma)\cap \{\sqrt{\kp H} t_\gamma (x)\leq 1\}}|\psi(x)|^2\,dx&{\rm if}~h=f_\gamma\,,\\
&\\
\displaystyle\int_{\omega(\gamma)\cap \{\sqrt{\kp H} \zeta_\gamma(x)\leq 1\}}|\psi (x) |^2\,dx&{\rm if}~h=g_\gamma\,.
\end{cases}
\end{equation}
In this way we infer from \eqref{exp1*} the following estimate
\begin{equation}
\int_{\omega(\gamma)} \Big( \big|(\nabla-i\kp H {\bf
A})(h\psi) \, (x)\big|^2-\kappa^2h(x)^2|\psi (x) |^2- 2 \hat \alpha^2\kappa Hh(x)^2|\psi (x) |^2\Big)\,dx\leq 
C\, \kappa H\, T(h)\,. \label{exp1}                                                  
\end{equation}
~\\
\subsection{Decay in the interior}~\\

Now we choose $$\gamma=\lambda =\frac\kappa H+\mu\,.
$$
Here $0<\mu<\mu_0$ and $\mu_0$ is sufficiently small such that $\mu_0+\frac1{c_1}<\beta_0\,$. This ensures that $\omega(\lambda)\not=\emptyset\,$.

We choose  in  \eqref{exp1}  the function $h=f_\lambda$, where $f_\lambda$ is the function introduced in \eqref{eq:f,g}. Note that $f_\lambda\psi\in H^1_0(\omega(\lambda))$. We may apply the result in Proposition~\ref{lem:qf}  to $\phi:=f_\lambda \psi$ and infer from
 \eqref{exp1}
$$
\int_{\omega(\lambda)}\Big(\big(1-C\kappa^{-\sigma(\alpha)})|B_0(x)|-2 \hat \alpha^2-\frac{\kappa}{H}\Big)f_\lambda^2\, |\psi|^2\,dx
\leq C\int_{\omega(\lambda)\cap \{\sqrt{\kp H} t_\lambda(x)\leq 1\}}|\psi(x)|^2\,dx\,.
$$
We then use that $|B_0(x)|\geq \lambda$  in $\omega(\lambda)$ and that $\lambda=\frac{\kappa}H+\mu\,$. Consequently,   for $0<\mu<\mu_0\,$, $0<\hat \alpha<\hat \alpha_0\,$, $\kappa\geq \kappa_0$\,, $\hat \alpha_0$ sufficiently small (for example $\hat \alpha_0^2 < \mu/4$)  and $\kappa_0$ sufficiently large
$$
\big(1-C\kappa^{-\sigma(\alpha)})|B_0(x)|- 2\hat \alpha^2-\frac{\kappa}{H}
\geq \frac{\mu}2\,.$$
Consequently, there exists a constant $C_\mu>0$ such that
$$
\begin{aligned}
\int_{\omega(\lambda)}f_\lambda(x)^2\, |\psi(x)|^2\,dx
&\leq C_\mu^{-1}\int_{\omega(\lambda)\cap \{\sqrt{\kp H} t_\lambda(x)\leq 1\}}|\psi (x) |^2\,dx\\
&\leq \frac{C}{\sqrt{\kappa H}}\quad\quad\quad {\rm by~}\eqref{eq:psi<1}\,.
\end{aligned}$$
Inserting this into \eqref{exp1} (with $h=f_\lambda$ and $T(f_\lambda)$ defined in \eqref{eq:T(h)}) achieves  the proof of Item~(1) in Theorem~\ref{thm:exp-dec}.

\subsection{Decay up to the boundary}

Now we prove Item~(2) in Theorem~\ref{thm:exp-dec}. Here we choose 
$$\gamma=\beta =\Theta_0^{-1}\,
\left(\frac\kappa H+\mu\right)\,.
$$
 Note that the estimate in Item~(2) of Theorem~\ref{thm:exp-dec} is trivially true if $\omega(\beta)=\emptyset\,$. So, we assume in the sequel that $\omega(\beta)\not=\emptyset\,$. This holds  if 
 $$H\geq c_1\kappa\,, \quad\, c_1> (\Theta_0\beta_1)^{-1}\,,$$ and $\mu$ is sufficiently small.

 We write \eqref{exp1} for $h=g_\beta\,$, where $g_\beta$ is introduced in \eqref{eq:f,g} and $T(g_\beta)$  in \eqref{eq:T(h)}. We apply  Proposition~\ref{lem:qf}  to $\phi:=g_\beta\psi$ and get
\begin{equation}\label{exp3}\int_{\omega(\beta)}\Big(\big(1-C\kappa^{-\sigma(\alpha)})\Theta_0|B_0(x)|-C\hat \alpha^2-\frac{\kappa}{H}\Big)g_\beta(x)^2|\psi(x)|^2\,dx\,
\leq C\int_{\omega(\beta)\cap \{\sqrt{\kp H} \zeta_\beta(x)\leq 1\}}|\psi(x)|^2\,dx\,.
\end{equation}
We decompose the integral over $\omega(\beta)$ as follows
$$\int_{\omega(\beta)}=\int_{\omega_{\rm int}(\beta)}+\int_{\omega_{\rm bnd}(\beta)}\,,$$
where
$$
\omega_{\rm int}(\beta)=\omega(\beta)\cap\big\{\sqrt{\kappa H}\,{\rm dist}(x,\partial\Omega)\geq 1\big\}\quad
{\rm and}\quad
\omega_{\rm bnd}(\beta)=\omega(\beta)\cap\big\{\sqrt{\kappa H}\,{\rm dist}(x,\partial\Omega)< 1\big\}\,.$$
From  \eqref{eq:t(x)}, we see that $\zeta_\beta(x)=t_\beta(x)$ and $f_\beta(x)=g_\beta(x)$ in $\omega_{\rm int}(\beta)$\,. Furthermore, from the definition of $\omega(\cdot)$ in \eqref{eq:om(l)}, we see that $\omega(\beta)\subset\omega(\lambda)$ and $t_\beta(x)\leq t_\lambda(x)$ on $\omega(\beta)$  if  $\beta\geq\lambda\,$. Hence, by the first item in Theorem~\ref{thm:exp-dec} (which is already proved for all $\hat \alpha\in(0,\hat \alpha_0)$),
\begin{equation}\label{exp3:int}
\int_{\omega_{\rm int}(\beta)}\Big|\big(1-C\kappa^{-\sigma(\alpha)})\Theta_0|B_0(x)|-2 \hat \alpha^2-\frac{\kappa}{H}\Big|\,g_\beta(x)^2\, |\psi(x)|^2\,dx\leq \frac{C}{\kappa}\,.
\end{equation}
Thus, we infer from \eqref{exp3} (and the bound $|\psi|\leq1$),
$$\int_{\omega_{\rm bnd}(\beta)}\Big(\big(1-C\kappa^{-\sigma(\alpha)})\Theta_0|B_0(x)|- 2 \hat \alpha^2-\frac{\kappa}{H}\Big)\, g_\beta(x)^2|\psi(x)|^2\,dx\leq \frac{C}{\kappa}
\,.$$
But,  in $\omega_{\rm bnd}(\beta)$, $\Theta_0|B_0(x)|\geq \frac\kappa H+\mu$\,,
 by definition of $\omega(\beta)$ and $\beta=\Theta_0^{-1}(\frac\kappa H+\mu)\,$. Thus, as long as $\hat \alpha$ is selected sufficiently small, we have
$$(1-C\kappa^{-\sigma(\alpha)})\, \Theta_0|B_0(x)|-2 \hat \alpha^2-\frac{\kappa}{H}\geq \frac\mu2\,,$$
and consequently, for some constant $\tilde C_\mu>0\,$,
$$\int_{\omega_{\rm bnd}(\beta)} g_\beta(x)^2|\psi(x)|^2\,dx\leq \frac{\tilde C_\mu}\kappa
\,.$$
We insert this estimate  and the one in \eqref{exp3:int} into \eqref{exp3} to get
$$\int_{\omega(\beta)} g_\beta(x)^2|\psi(x)|^2\,dx\leq \frac{\tilde C_\mu+C}\kappa
 \,.$$
Finally, by inserting this estimate into \eqref{exp1} (with $h=g_\beta$ and $T(g_\beta)$ defined in \eqref{eq:T(h)}), we  finish the proof of Item~(2) in Theorem~\ref{thm:exp-dec}.

\section{Surface energy}\label{sec:surf}
The analysis of surface superconductivity starts with the work of St.~James-de\,Gennes~\cite{SJ-dG}, who studied this phenomenon on the ball. In the last two decades, many papers adressed   the boundary concentration of the Ginzburg-Landau order parameter for general $2D$ and $3D$ samples in the presence of a constant magnetic field. We refer the reader to \cite{AH, CR, FoHe04, FK-am, FKP-jmpa, FHP, LuPa1, Pa02}.

In this section, we  study surface superconductivity in non-uniform magnetic fields. Our presentation not only generalizes the results known for the constant field case, but also  provides local estimates and new proofs, see Theorems~\ref{thm:ub-l4} and \ref{thm:lb-psi}\,. The most notable novelty in the proofs is that we do not use the $L^\infty$ elliptic estimates. 

\subsection{The surface energy function}\label{ss4.1}

In this subsection, we give the definition of the continuous function  $E_{\rm surf}:[1,\Theta_0^{-1}]\to (-\infty,0]$ introduced by X.B. Pan in \cite{Pa02} and which appeared after \eqref{eq:1.15*} and in Theorem \ref{thm:surface-sc}.  $\Theta_0$ is as before the de\,Gennes constant introduced in \eqref{eq:Theta0} with property \eqref{eq:Theta_0a}.

 For $b\in[1,\Theta_0^{-1}]\,$ and $R>0\,$, we consider the reduced Ginzburg-Landau  functional,
\begin{equation}\label{eq-hc2-redGL}
\mathcal V(U_R)\ni \phi \mapsto \mathcal E_{b,R}(\phi)=\int_{U_R}\left(b|(\nabla_{(\sigma,\tau)}+i\tau\fb)\phi|^2-|\phi|^2+\frac12|\phi|^4\right)\,d\sigma d\tau\,,
\end{equation}
where $\fb=(1,0)$ and $U_R$ is the domain,
\begin{equation}\label{eq-hc2-Uell}
U_R=(-R,R)\times(0,+\infty)\,,
\end{equation}
and 
\begin{equation}\label{eq-hc2-Confspace}
\mathcal V(U_R)=\{u\in L^2(U_R)~:~(\nabla_{(\sigma,\tau)}+i\tau\fb)u\in
L^2(U_R)~,~u(\pm R,\cdot)=0\,\}\,.
\end{equation}
 We  introduce the following ground state energy,
\begin{equation}\label{eq-hc2-d(ell)}
d(b,R)=\inf\{\mathcal E_{b,R}(\phi)~:~\phi\in\mathcal V(U_R)\}\,.
\end{equation}
In \cite{Pa02}, it is proved that, for all $b\in[1,\Theta_0^{-1}]\,$, there exists $E_{\rm surf}(b)\in(-\infty,0]$ such that
\begin{equation}\label{eq:Pa02}
E_{\rm surf}(b)=\lim_{R\to\infty}\frac{d(b,R)}{2R}\,.
\end{equation}
The surface energy function $\Es(\cdot)$ can be described by a simplified $1D$ problem as well   (cf. \cite{AH,  FHP} and finally \cite{CR} for the optimal result).   We collect some properties of $\Es(\cdot)$:
\begin{itemize}
\item  $\Es(\cdot)$ is a continuous and increasing function (cf. \cite{FKP-jmpa})\,;
\item $\Es(\Theta_0^{-1})=0$  and 
$\Es(b)<0$ for all $b\in[1,\Theta_0^{-1})$ (cf. \cite{FoHe04}). 
\end{itemize}

The next theorem gives the existence of some minimizer with good properties  (cf. \cite[Theorems~4.4 \& 5.3]{Pa02}):
\begin{thm}\label{thm-hc2-Pa02}
There exist positive constants $R_0$ and $M$  such that, for all $b\in[1,\Theta_0^{-1})$ and $R\geq R_0$:
\begin{enumerate}
\item The functional \eqref{eq-hc2-redGL} has a
minimizer $u_R$ in $\mathcal V(U_R)$  with the following properties: 
\begin{enumerate}
\item $u_R\not\equiv0$\,;
\item 
$\|u_R\|_\infty\leq1$\,;
\item 
$$ 
\frac 1R \, \int_{U_R\cap\{\tau\geq 3\}}\frac{\tau^2}{(\ln
\tau)^2}\left(|(\nabla_{(\sigma,\tau)}+i\tau\fb )u_R|^2+|u_R(\sigma, \tau)|^2+\tau^2|u_R(\sigma,\tau)|^4\right)\,d\sigma d\tau  \leq
M \,.$$
\end{enumerate}
\item The surface  energy function  $E_{\rm surf}(b)$ satisfies 
$$E_{\rm surf}(b)\leq \frac{d(b,R)}{2R}\leq E_{\rm surf}(b)+\frac{M}{R}\,.$$
\end{enumerate}
\end{thm}

The upper bound in  Item (2) above  results from a property of 
superadditivity of $d(b,R)$, see \cite[Eq.~(5.4)]{Pa02}.  The lower bound in  Item (2) above is not explicitly  mentioned in \cite{Pa02}, but its derivation is easy \cite[Proof of Thm~2.1, Step~2, p. 351]{FK-cpde} and can be sketched in the following way.
Let $R>0$ and $n\in\mathbb N$. Let $u_R \in H^1_0(U_R)$ be a minimizer of the functional in \eqref{eq-hc2-redGL}. We extend 
$u_R$ to a function in $H^1_0(U_{(2n+1)R})$ by periodicity as follows
$$u_R(x_1+2R,x_2)=u_R(x_1,x_2)\,.$$
Consequently, 
$$d(b,(2n+1)R)\leq \mathcal E_{b,(2n+1)R}(u_R)=(2n+1)d(b,R)\,.$$
Dividing both sides of the preceding inequality by $2(2n+1)R$ and sending $n$ to $+\infty\,$, we get
$$\Es(b)\leq \frac{d(b,R)}{2R}\,.$$

\subsection{Boundary coordinates}

The analysis of the boundary effects is performed in specific coordinates  valid in a tubular
neighborhood of $\partial\Omega$. We call these coordinates {\it boundary coordinates}.  For more details on these
coordinates, see for instance \cite[Appendix F]{FH-b}.

For a sufficiently small $t_0>0$, we introduce the open set
$$\Omega(t_0)=\{x\in\mathbb R^2~:~{\rm
  dist}(x,\partial\Omega)<t_0\}\,.$$
In the sequel, let $x_0\in\partial\Omega$ be a fixed point. Let $s\mapsto\gamma_{x_0}(s)$ be the  parametrization of $\partial\Omega$
by arc-length such that $\gamma_{x_0}(0)=x_0$. Also, let $\nu(s)$ be the unit inward normal of
$\partial\Omega$ at
  $\gamma_{x_0}(s)$. The orientation of $\gamma_{x_0}$ is selected in the counter clock-wise direction, hence
  $${\rm det}\Big(\gamma_{x_0}'(s),\nu(s)\Big)=1\,.$$
Define the transformation
\begin{align}\label{eq:19}
\Phi_{x_0}:
\left[-\frac{|\partial\Omega|}2,\frac{|\partial\Omega|}2\right)\,\times
(0,t_0)\ni(s,t)\mapsto \gamma_{x_0}(s)+t\nu(s)\in \Omega(t_0)\,.
\end{align}
We may choose $t_0$ sufficiently  small (independently from the choice of the point $x_0\in\partial\Omega$) such that the transformation 
in \eqref{eq:19} is a diffeomorphism. The  Jacobian of this transformation is $|D\Phi_{x_0}|=1-tk
 (s)$, where $k$ denotes the curvature of $\partial\Omega$. For $x\in\Omega(t_0)$, we put
$$\Phi^{-1}_{x_0}(x)=(s(x),t(x))\,.$$
In particular, we get the explicit formulae  \begin{equation} \label{eq-hc2-t(x)}
t(x)={\rm dist}(x,\partial\Omega)\quad{\rm and}\quad s(x_0)=0\,.\end{equation}

Using  $\Phi_{x_0}$, we may associate to any
function $u\in L^2(\Omega)$, a function $\widetilde u=T_{\Phi_{x_0}} u$ defined in $
[-\frac{|\partial\Omega|}2,
\frac{|\partial\Omega|}2)\,\times\,(0,t_0)$ by,
\begin{equation}\label{VI-utilde}
\widetilde u(s,t)=u(\Phi_{x_0}(s,t))\,.
\end{equation}
Also, for every vector field $\Ab\in H^1(\Omega)$, we assign the vector field  $$\tilde\Ab(s,t)=\Big(\tilde\Ab_1(s,t),\tilde\Ab_2(s,t)\Big)$$ with 
\begin{equation}\label{eq:tildeA} 
\left\{
\begin{array}{rl}
\tilde \Ab_1(s,t) & = a(s,t)\Ab\Big( \Phi_{x_0}(s,t)\Big)\cdot\gamma_{x_0}'(s)\,,\\ 
\tilde \Ab_2 (s,t)& = \Ab \Big(
\Phi_{x_0}(s,t)\Big)\cdot\nu(s)\,,
\end{array}
\right.
\end{equation}
and
$$
  a(s,t) =1-t\,k(s) \,.
$$
The following change of variable formulas hold.
\begin{proposition}\label{hc2-App:transf}
For $u\in H^1(\Omega)$ and $\Ab\in H^1(\Omega;\R^2)$, we have:
\begin{multline}\label{App:qfstco}
\int_{\Omega( t_0)}\left|(\nabla-i\Ab)u\right|^2dx =
\int_{0}^{t_0}\int_{-\frac{|\partial\Omega|}2}^{\frac{|\partial\Omega|}2}
\left[ [a(s,t)]^{-2}|(\partial_s-i\tilde
\Ab_1)\widetilde u|^2+|(\partial_t-i\tilde \Ab_2)\widetilde
u|^2\right]a(s,t)\,dsdt\,,
\end{multline}
and
\begin{equation}\label{App:nostco}
\int_{\Omega(t_0)}
|u(x)|^2\,dx=\int_{0}^{t_0}\int_{-\frac{|\partial\Omega|}2}^{\frac{|\partial\Omega|}2}
 |\widetilde u(s,t)|^2a(s,t)\,dsdt\,.
\end{equation}
\end{proposition}

Recall the vector field $\Ab_0$ introduced in \eqref{eq:A0}.
Up to a gauge transformation, the vector field $\Ab_0$  admits a useful (local) representation in the coordinate system $(s,t)$.

For $x_0\in\partial\Omega$ and $\ell\in(0,t_0)$,  we introduce the set $ V_{x_0}(\ell) \subset\Omega(t_0)$  as follows:
\begin{equation}\label{eq-hc2-U12}
V_{x_0}(\ell)=\Phi_{x_0}\Big((-\ell,\ell)\times(0,\ell)\Big)\,.
\end{equation}
\begin{lem}\label{eq-hc2-gaugeT}  
There exists $r_0>0$ such that, for any $x_0$ in $\partial\Omega\,$,  there exists $g_{x_0}$ in \break $ C^\infty( (-2r_0\,,\,2r_0)\times(0\,,r_0))$ such that 
$$\tilde\Ab_0(s,t) -\nabla g_{x_0}(s,t) =
\left(-t+k(s)\frac{t^2}2,0\right)\quad{\rm
in}~(-2r_0\,,\,2r_0)\times(0\,,r_0)\,.
$$
Here $\tilde\Ab_0$ is the vector field associated with $\Ab_0$ by the formulas in \eqref{eq:tildeA} and one can take  $r_0=\min(t_0,\frac{|\partial\Omega|}{4})$. 
\end{lem}

For the proof of Lemma~\ref{eq-hc2-gaugeT}, we refer to \cite[Proof of Lem.~F.1.1]{FH-b}. Note that Lemma~F.1.1 in \cite{FH-b} is announced for a more general setting.

We will use Lemma~\ref{eq-hc2-gaugeT} to estimate the following Ginzburg-Landau energy  of $u$,
\begin{equation}\label{eq:GL-loc}
\mathcal G_0\big(u,\Ab_0;V_{x_0}(\ell)\big)=\int_{V_{x_0}(\ell)}\Big(|(\nabla-ih_{\rm ex}\Ab_0)u|^2-\kappa^2|u|^2+\frac{\kappa^2}2|u|^4\Big)\,dx\,.
\end{equation}
\begin{lem}\label{lem:est-E0}
There exist constants $C>0$, $\ell_0>0$ and $\kappa_0>0$ such that, 
for all  $x_0\in\partial\Omega$, $\ell\in(0,\ell_0)$, $\kappa\geq \kappa_0\,$, $\kappa^2\leq h_{\rm ex}\leq\Theta_0^{-1}\kappa^2$\,,
and  $u\in H^1_0(V_{x_0}(\ell))\cap L^\infty(V_{x_0}(\ell))$ satisfying $\|u\|_\infty\leq 1$\,,
the following two inequalities hold:
\begin{equation}\label{eq:lbE0}
\mathcal G_0\big(u,\Ab_0;V_{x_0}(\ell)\big)\geq {2} \frac{\kappa^2\ell}{\sqrt{h_{\rm ex}}}\Es\left(\frac{h_{\rm ex}}{\kappa^2}\right)
-C\kappa\ell\Big(\ell+\kappa^3\ell^4+\kappa\ell^2\Big)\,,
\end{equation} 
and
\begin{equation}\label{eq:upE0}
\mathcal G_0\big(u,\Ab_0;V_{x_0}(\ell)\big)\leq (1+C\ell) \frac{\kappa^2}{h_{\rm ex}}\mathcal E_{h_{\rm ex}/\kappa^2,\sqrt{h_{\rm ex}}\,\ell}
(\widetilde v)+C\kappa\ell\Big(\kappa^3\ell^4+\kappa\ell^2\Big)\,.
\end{equation}
where $\mathcal E_{\cdot,\cdot }$ is the functional introduced in \eqref{eq-hc2-redGL} and
$$\widetilde v(\sigma,\tau)=\exp\left(-ih_{\rm ex}\,g_{x_0}\Big(\frac{\sigma}{\sqrt{h_{\rm ex}}},\frac{\tau}{\sqrt{h_{\rm ex}}}\Big)\right)\widetilde u\Big(\frac{\sigma}{\sqrt{h_{\rm ex}}},\frac{\tau}{\sqrt{h_{\rm ex}}}\Big)\,.$$
Here $\widetilde u$ is the function associated with $u$ by \eqref{VI-utilde} and  $g_{x_0}$ is  introduced in Lemma~\ref{eq-hc2-gaugeT}.
\end{lem}
\begin{proof}
Using  Proposition~\ref{hc2-App:transf}  and the assumptions on $u$, we may write, for two positive constants $C_0,C$ and for all $0<\ell<\min\big(\frac12 C_0^{-1},t_0\big)$,
$$\mathcal G_0\big(u,\Ab_0;V_{x_0}(\ell)\big)\geq (1-C_0\ell)\int_0^{\ell}\int_{-\ell}^{\ell}\Big(|(\nabla-ih_{\rm ex}\widetilde\Ab_0)\widetilde u|^2-\kappa^2|\widetilde u|^2+\frac{\kappa^2}2|\widetilde u|^4\Big)\,dsdt-C\kappa^2\ell^3\,.$$

Let $g:=g_{x_0}$ be the function defined in Lemma~\ref{eq-hc2-gaugeT} and 
$\widetilde w(s,t)=e^{-i h_{\rm ex}g(s,t)}\widetilde u(s,t)\,$. Using the Cauchy-Schwarz inequality, we get the existence of $C>0$ such that
$$\mathcal G_0\big(u,\Ab_0;V_{x_0}(\ell)\big)\geq(1-2C_0\ell)\int_0^{\ell}\int_{-\ell}^{\ell}\Big(|(\nabla+ih_{\rm ex} t\fb)\widetilde w|^2-\kappa^2|\widetilde w|^2+\frac{\kappa^2}2|\widetilde w|^4\Big)\,dsdt-C\kappa^4\ell^5-C\kappa^2\ell^3\,.$$
Here $\fb=(1,0)$. We apply the change of variables $(\sigma,\tau)=(\sqrt{h_{\rm ex}}\,s,\sqrt{h_{\rm ex}}\,t)$ and
$\widetilde v(\sigma,\tau)=\widetilde w(s,t)$ to get
$$\mathcal G_0(u,\Ab_0;V_{x_0}(\ell))\geq(1-2C_0\ell)\frac{\kappa^2}{h_{\rm ex}}\mathcal E_{h_{\rm ex}/\kappa^2,R}
({\widetilde v})-C\kappa^4\ell^5-C\kappa^2\ell^3\,,$$
where $R=h_{\rm ex}^\frac 12 \,\ell$ and $\mathcal E_{h_{\rm ex}/\kappa^2,R}$ is the functional introduced in \eqref{eq-hc2-redGL} for $b= {h_{\rm ex}}/\kappa^2\,$.

  Note that we extended $\widetilde v$ by $0$, which is possible because $u\in H^1_0(V_{x_0}(\ell))$. Using the  second Item   in Theorem~\ref{thm-hc2-Pa02} and the assumption $C_0 \ell < \frac 12\,$, 
we get
$$\mathcal G_0(u,\Ab_0;V_{x_0}(\ell))\geq 2 (1-2C_0\ell)\, \frac{\kappa^2}{h_{\rm ex}}\,(h_{\rm ex}^\frac 12 \,\ell)\,\Es\left(\frac{h_{\rm ex}}{\kappa^2}\right)-C\kappa^4\ell^5-C\kappa^2\ell^3\,.$$
This proves the lower bound  \eqref{eq:lbE0} in Lemma~\ref{lem:est-E0}\,.\\
Similarly, using Lemma~\ref{eq-hc2-gaugeT}, the Cauchy-Schwarz inequality on the kinetic term
and a change of variables,  we get the upper bound \eqref{eq:upE0}  of Lemma~\ref{lem:est-E0}\,.
\end{proof}
~

\subsection{Existence of surface superconductivity}~

The proof of Theorem~\ref{thm:surface-sc} follows from the exponential decay stated in Theorem~\ref{thm:exp-dec} and the following  result:

\begin{thm}\label{thm:surface-sc*}
Suppose that Assumption~\ref{ass:mf*} holds and that $b>\beta_0^{-1}$, where $\beta_0$ is the constant introduced in
\eqref{eq:ep0}. There exists $\rho\in(0,1)$ such that the following is true.

Let $x_0\in\partial\Omega$ such that  $ \frac 1b < |B_0(x_0)|<\frac{1}{ \Theta_0 b}\,$. 
If $(\psi,\Ab)_{\kappa,H}$ is a minimizer of the functional in \eqref{eq-3D-GLf} for $H=b\kappa\,$, then
\begin{equation}\label{eq:1.13}
\lim_{\kappa\to +\infty}\left(2\kappa^{1+\rho}\int_{V_{x_0}(\kappa^{-\rho})} |\psi(x)|^4\,dx\right)=-2\sqrt{\frac1{b|B_0(x_0)|}}\Es\big(b|B_0(x_0)|\big)>0\,,
\end{equation}
and
\begin{equation}\label{eq:1.14}
\lim_{\kappa\to +\infty}\Big(2\kappa^{\rho-1}\mathcal E\big(\psi,\Ab;V_{x_0}(\kappa^{-\rho})\big)\Big)=\sqrt{\frac1{b|B_0(x_0)|}}\Es\big(b|B_0(x_0)|\big)<0\,.
\end{equation}
\end{thm}

The proof of Theorem~\ref{thm:surface-sc*} will follow from the upper bound in Theorems~\ref{thm:ub-l4} and \ref{thm:lb-psi} below. 

\begin{rem}\label{rem:unif-conv}
Let $\epsilon\in(1,\Theta_0^{-1}-1)\,$. The convergence in  \eqref{eq:1.13} and \eqref{eq:1.14}  is uniform with respect to  $x_0\in\{1+\epsilon\leq b|B_0|<\Theta_0^{-1}\}\cap\partial\Omega$. This is precisely stated in Theorems~\ref{thm:ub-l4} and \ref{thm:lb-psi}. 
\end{rem}
~

\subsection{Sharp upper bound on the $L^4$-norm}~

In this subsection, we will prove:

\begin{thm}\label{thm:ub-l4} 
Suppose that  $B_0\in C^{0,\alpha}(\overline\Omega)$ for some $\alpha\in(0,1)$, $\rho\in(\frac{3}{3+\alpha},1)$ and
\begin{equation*}
 b\geq\beta_0^{-1}\,,\, \mbox{ with } \beta_0:=\sup_{x\in\overline\Omega}|B_0(x)|>0\,.
\end{equation*} 
There exist $\kappa_0>0$,  a function $ {\rm r}:[\kappa_0,+\infty)\to\R_+$ 
such that $\lim_{\kappa\to +\infty}{\rm r}(\kappa)=0$ and, 
for all $\kappa\geq\kappa_0$\,, for all critical point  $(\psi,\Ab)_{\kappa,H}$ of the functional in \eqref{eq-3D-GLf} with $H=b\kappa$\,, and all  $x_0\in \partial\Omega$ satisfying  $$1\leq b\, |B_0(x_0)|  < \Theta_0^{-1}\,,$$
the inequality 
$$
\frac1{2\ell}\int_{V_{x_0}(\ell)}|\psi(x)|^4\,dx\leq -2\kappa^{-1}\sqrt{\frac{1}{b|B_0(x_0)|}}\,\Es\Big(b\, |B_0(x_0)|\Big)
+\kappa^{-1} \,{\rm r}(\kappa)\,,
$$
holds with 
$$\ell = \kappa^{-\rho}
\text{ and }\quad 
V_{x_0}(\ell) \text{ is defined in \eqref{eq-hc2-U12}}.$$
\end{thm}
\begin{proof}~\\
The proof is reminiscent of the method  used by the second author in \cite[Sec.~4]{K-CMB} (see also \cite{KN-3D}). We  assume that $B_0(x_0)>0$. The case where $B_0(x_0)<0$ can be treated in the same manner by applying the transformation $u\mapsto \overline{u}$.

Let $\sigma\in(0,1)$ and $\ell=\kappa^{-\rho}$ as  in the statement of Theorem~\ref{thm:ub-l4}\,.
Let  $f$ be a smooth function  satisfying, 
\begin{equation}\label{eq:f*}
f=1\quad{\rm in~} V_{x_0}(\ell),\quad 0\leq f\leq 1{~\rm and~}|\nabla f|\leq \frac{C}{\sigma\ell}\quad {\rm in~}V_{x_0}\big((1+\sigma)\ell\big)\,.
\end{equation}
The function $f$ depends  on the parameters $x_0,\ell,\sigma$ but the constant $C$ is independent of these parameters.
We will estimate the following local energy 
\begin{equation}\label{eq:gse-loc}  
\mathcal E_1(f\psi,\Ab):=\mathcal E_1\big(f\psi,\Ab;  V_{x_0}((1+\sigma)\ell)\big)\,,
\end{equation}
where, for an open set $\mathcal V\subset \Omega$, 
\begin{equation}\label{eq:GLV}
\begin{aligned}
&\mathcal E_1(u,\Ab; \mathcal V):=\int_{\mathcal V}\left(|(\nabla-i\kappa H\Ab)u |^2-\kappa^2|u|^2+\frac{\kappa^2}2|u|^4\right)\,dx\,,\\
&\mathcal E_2(u,\Ab; \mathcal V):=\int_{\mathcal V}\left(|(\nabla-i\kappa H\Ab)u |^2-\kappa^2|u|^2+\frac{\kappa^2}2|u|^4\right)\,dx\\
&\hskip3.5cm+(\kappa H)^2\int_\Omega|\curl\Ab-B_0|^2\,dx\,.
\end{aligned}
\end{equation}

Since $(\psi,\Ab)$ is a solution of \eqref{eq:GL}, an integration by parts  yields (cf. \cite[Eq.~(6.2)]{FK-am}),
\begin{equation}\label{eq:decomp*}
\mathcal E_1(f\psi,\Ab)=\kappa^2\int_{V_{x_0}((1+\sigma)\ell)}f^2\left(-1+\frac12f^2\right)|\psi|^4\,dx+\int_{V_{x_0}((1+\sigma)\ell)}|\nabla f|^2|\psi|^2\,dx\,.
\end{equation}
Since $f=1$ in $V_{x_0}(\ell)$ and  $-1+\frac12f^2\leq -\frac12$ in $V_{x_0}((1+\sigma)\ell)$, we may write
$$\int_{V_{x_0}((1+\sigma)\ell)}f^2\left(-1+\frac12f^2\right)|\psi|^4\,dx\leq -\frac12\int_{V_{x_0}(\ell)}|\psi|^4\,dx\,.$$
We estimate the integral in \eqref{eq:decomp*} involving $|\nabla f|$  using  \eqref{eq:f*} and  $ \big|{\rm supp\,}f\big|\leq C\sigma\ell^2$, where  $\big|{\rm supp\,}f\big|$ denotes the area of the support of $f\,$. 
In this way, we infer from \eqref{eq:decomp*},
\begin{equation}\label{eq:decomp**}
\mathcal E_1(f\psi,\Ab)\leq -\frac{\kappa^2}2\int_{V_{x_0}(\ell)}|\psi|^4\,dx+C\sigma^{-1}\,.
\end{equation}
Now we write a lower bound  for this energy. We may find a real-valued function\break$w\in C^{2,\alpha}(\overline{V_{x_0}((1+\sigma)\ell})$ 
such that
\begin{multline*}
\mathcal E_1(f\psi,\Ab)\geq 
\int_{V_{x_0}((1+\sigma)\ell)}\Big(
(1-C\ell^\delta)|(\nabla-i\kappa H B_0(x_0)\Ab_0)(e^{-i\kappa Hw}f\psi) |^2-\kappa^2|f\psi|^2+\frac{\kappa^2}2|f\psi|^4\Big)\,dx\\
-C\kappa^2\Big(\ell^{2\gamma-\delta}+\kappa^2\ell^{2+2 \alpha -\delta}\Big)\int_{V_{x_0}((1+\sigma)\ell)}|f\psi|^2\,,\end{multline*}
where $\gamma\in(0,1)$ is a constant whose choice will be specified later and $\delta >0\,$. \\The details of   these computations are given in \eqref{eq:ke-1} and \eqref{eq:ke-2}.\\
 From now on we choose $\delta=\alpha\,$, use the lower bound in Lemma~\ref{lem:est-E0} and the assumption that $H=b\kappa$ to write
\begin{multline*}
\mathcal E_1(f\psi,\Ab)\geq 
2(1-C\ell^\alpha)\kappa(1+\sigma)\ell\,  \sqrt{\frac1{b|B_0(x_0)|}}\,\Es\big(b|B_0(x_0)|\big)\\
-C\kappa\ell(\ell+\kappa^3\ell^4+\kappa\ell^2)-C\kappa^2\Big(\ell^\alpha+\ell^{2\gamma-\alpha}+\kappa^2\ell^{2+\alpha}\Big)\int_{V_{x_0}((1+\sigma)\ell)}|f\psi|^2\,.
\end{multline*}
Using the bound $\|f\psi\|_\infty\leq 1\,$, we get further
\begin{multline}\label{eq:lb-loc-en}
\mathcal E_1(f\psi,\Ab)\geq 
2(1-C\ell^\alpha)\kappa(1+\sigma)\ell \,  \sqrt{\frac1{b|B_0(x_0)|}}\, \Es\big(b|B_0(x_0)|\big)\\
-C\kappa\ell\Big(\ell+\kappa\ell^{1+\alpha}+\kappa\ell^{1+2\gamma-\alpha}+\kappa^3\ell^{3+\alpha}\Big)\,.
\end{multline}
To optimize the remainder, we choose $\gamma=\alpha\,$. Our assumption
$$\ell=\kappa^{-\rho}\quad{\rm with}\quad (1+\alpha)^{-1}<3(3+\alpha)^{-1}<\rho<1$$ yields that the function
$$\Sigma(\kappa,\ell) :=\ell^\alpha+\ell+\kappa\ell^{1+\alpha}+\kappa^3\ell^{3+\alpha}$$
tends, with $\ell=\kappa^{-\rho}$, to $0$ as $\kappa \rightarrow +\infty$\,.\\
Now, coming back to \eqref{eq:decomp**}, we find
$$
2\kappa(1+\sigma)\ell \sqrt{\frac1{b|B_0(x_0)|}}\Es\big(b|B_0(x_0)|\big)-C\,\kappa\,\ell\,\Sigma(\kappa,\ell) \leq -\frac{\kappa^2}2
\int_{V_{x_0}(\ell)}|\psi|^4\,dx+C\sigma^{-1}\,.
$$
We rearrange the terms in this inequality,  divide by $\kappa^2 \ell$,  and  choose $\sigma=\kappa^{\frac12(\rho-1)}$. In this way, we get the upper bound in Theorem~\ref{thm:ub-l4} with, for some constant $C>0$, 
$$
 {\rm r}(\kappa)=C\, \Big(\Sigma(\kappa, \kappa^{-\rho})+ \kappa^{\frac 12(\rho-1)} \Big)\,.$$
\end{proof}

\subsection{Sharp Lower bound on the $L^4$-norm}~ 

In this subsection, we will  prove the asymptotic optimality of the upper bound  established in Theorem~\ref{thm:ub-l4}   by giving a lower bound with the same asymptotics.

We remind the reader of the definition of the domain $V_{x_0}(\ell)$ in \eqref{eq-hc2-U12} and
the local energy $\mathcal E_1\big(\psi,\Ab;\mathcal V\big)$ introduced in \eqref{eq:GLV}.

\begin{thm}\label{thm:lb-psi}
Let $1<\epsilon<\Theta_0^{-1}-1$, $\frac{3}{3+\alpha}<\rho<1$ and $1-\rho<\delta<1$ be constants. Under the assumptions of Theorem \ref{thm:ub-l4}, there exist $\kappa_0>0$,  a function $ {\rm \hat r}:[\kappa_0,+\infty)\to\R_+$ 
such that $\lim_{\kappa\to +\infty}{\rm \hat r}(\kappa)=0$ and, 
for all $\kappa\geq\kappa_0$\,, for all minimizer  $(\psi,\Ab)_{\kappa,H}$ of the functional in \eqref{eq-3D-GLf} with $H=b\kappa$\,, and all  $x_0\in \partial\Omega$ satisfying  $$1+\epsilon\leq b|B_0(x_0)|  < \Theta_0^{-1}\,,$$
the two inequalities
$$
\frac1{2\ell}\int_{V_{x_0}(\ell)}|\psi(x)|^4\,dx\geq -2\kappa^{-1}\sqrt{\frac{1}{b|B_0(x_0)|}}\,\Es\Big(b|B_0(x_0)|\Big)
-\kappa^{-1} \,{\rm \hat r} (\kappa)\,,
$$
$$
\left|\frac1{2\ell}\mathcal E_1\big(\psi,\Ab;V_{x_0}((1+\sigma)\ell)\big)-\kappa \sqrt{\frac{1}{b|B_0(x_0)|}}\,\Es\Big(b|B_0(x_0)|\Big)\right|
\leq \kappa \,{\rm \hat r} (\kappa) \,,
$$
hold,  with $\ell=\kappa^{-\rho}$ 
 and $\sigma=\kappa^{-\delta}$.
\end{thm}
\begin{rem}\label{rem:lb-psi*}
Let $c_2>c_1>0$ be fixed constants. The conclusion in Theorem~\ref{thm:lb-psi} remains true if $\ell$ satisfies
$$c_1\kappa^{-\rho}\leq \ell\leq c_2\kappa^{-\rho}\,.$$
\end{rem}

\begin{proof}[Proof of Theorem~\ref{thm:lb-psi}]
In the sequel, $\sigma\in(0,1)$ will be selected as a negative power of $\kappa\,$, $\sigma=\kappa^{-\delta}$ for a suitable constant 
$\delta\in(0,1)$. As  the proof of Theorem \ref{thm:ub-l4}, we can assume that $B_0(x_0)>0\,$. 
The proof of the lower bound in Theorem~\ref{thm:lb-psi} will be done in four steps.\\

{\bf Step~1: Construction of a trial function.}
 
The construction of the trial function here is reminiscent of that by Sandier-Serfaty in the study of bulk superconductivity (cf. \cite{SS02}).   Define the function
\begin{equation}\label{eq:u}
u(x)= \mathbf 1_{V_{x_0}((1+\sigma)\ell)}(x) \chi\left(\frac{t(x)}{\ell}\right)\, \exp\Big(i\kappa H w(x)\Big)\,v_R\circ \Phi_{x_0}^{-1}(x)
+\eta_\ell(x)\psi(x)\quad (x\in\Omega)\,.
\end{equation}
Here $V_{x_0}(\cdot)$ is introduced in \eqref{eq-hc2-U12}, $t(x)={\rm dist}(x,\partial\Omega)$, $\Phi_{x_0}$ is the coordinate transformation defined in \eqref{eq:19}, 
\begin{equation}\label{eq:vR}
v_R(s,t)=\exp\Big(i\kappa H g_{x_0}(s,t) \Big) u_R\big(s\sqrt{B_0(x_0)\kappa H},t\sqrt{B_0(x_0)\kappa H}\big)\,,
\end{equation}
\begin{equation}\label{eq:R}
R=(1+\sigma)\ell \sqrt{{B_0(x_0)}\kappa H}\,,
\end{equation}
and  (cf. \eqref{eq-hc2-redGL})
$$
u_R(\cdot) \mbox{ is a minimizer of the reduced functional } \mathcal E_{bB_0(x_0),R}(\cdot)\,.
$$ 
The function $g_{x_0}(s,t)$ satisfies the following identity in $\big(-2\ell\;,2\ell\,\big)\times(0,\ell)$ (cf. Lemma~\ref{eq-hc2-gaugeT}),

$$\tilde\Ab_0(s,t)-\nabla g_{x_0}(s,t)=\Big(-t+\frac{t^2}2 k(s),0\Big)\,.$$
The function $\chi\in C^\infty([0,\infty))$ satisfies
$$\chi=1~{\rm in~}[0,1/2]\,,\quad \chi=0~{\rm in~}[1,\infty)\,,\mbox{ and }\, 0\leq\chi\leq 1\,.$$
The function $\eta_\ell$ is a smooth function satisfying
$$\eta_\ell(x)=0~{\rm in~}V_{x_0}((1+\sigma)\ell)\,,\quad \eta_\ell(x)=1~{\rm in~}\Omega\setminus V_{x_0}((1+2\sigma)\ell)\,,\quad 0\leq \eta_\ell(x)\leq 1~{\rm in~}\Omega\,,$$
and
$$|\nabla\eta_\ell(x)|\leq C\sigma^{-1}\ell^{-1}\quad{\rm in~}\Omega\,,$$
for some constant $C>0$\,.

Finally, the function $w$ is the sum of two real-valued $C^{2,\alpha}$-functions $w_1$ and $w_2$  in $V_{x_0}((1+\sigma)\ell)$  
and satisfying the following estimates
\begin{equation}\label{eq:w}
|\Ab(x)-\Fb(x)-\nabla w_1(x)|\leq \frac{C}{\kappa}\ell^\alpha\quad{\rm and}\quad
|\Fb(x)-B_0(x_0)\Ab_0(x)-\nabla w_2(x)|\leq C\ell^{1+\alpha}\quad{\rm in~}V_{x_0}((1+\sigma)\ell)\,.
\end{equation}
By  Proposition~\ref{lem:Att}, we simply define $w_1(x)=(x-x_0)\cdot\big(\Ab(x_0)-\Fb(x_0)\big)$. 
 The fact that the vector field $\Ab_0(x)$ is gauge equivalent to $\Ab_0(x-x_0)$ and Lemma~\ref{lem:Gauge-Att} ensure the existence of $w_2$.

We decompose the energy $\mathcal E(u,\Ab)$ as follows
\begin{equation}\label{eq:E=E1+E2}
\mathcal E(u,\Ab)=\mathcal E_1(u,\Ab)+\mathcal E_2(u,\Ab)\,,
\end{equation}
where
\begin{equation}\label{eq:E1,E2}
\mathcal E_1(u,\Ab)=\mathcal E_1\Big(u,\Ab;V_{x_0}((1+\sigma)\ell)\Big)\quad{\rm and}\quad
\mathcal E_2(u,\Ab)=\mathcal E_2\Big(u,\Ab;\Omega\setminus V_{x_0}((1+\sigma)\ell)\Big)
\end{equation}
are introduced in \eqref{eq:GLV}.

{\bf Step~2: Estimating $\mathcal E_1(u,\Ab)$.}

Using the Cauchy-Schwarz inequality and the estimates in \eqref{eq:w}, we get
$$\mathcal E_1(u,\Ab)\leq {(1+\ell^\alpha)}\, \mathcal E_1\Big( e^{-i\kappa H w}\,u,B_0(x_0)\Ab_0\Big)+C\Big(\kappa^2\ell^{2+\alpha} +\kappa^4\ell^{4+\alpha}\Big)\,.
$$
For estimating the term $\mathcal E_1\Big( e^{-i\kappa H w}\,u,B_0(x_0)\Ab_0\Big)$, we write
$$\mathcal E_1\Big( e^{-i\kappa H w}\,u,B_0(x_0)\Ab_0\Big)=\mathcal G_0\Big( e^{-i\kappa H w}\,u,h_{\rm ex}\Ab_0;V_{x_0}(\tilde\ell)\Big)\,,$$
where
$$\tilde\ell=(1+\sigma)\ell\,,\quad h_{\rm ex}=\kappa H B_0(x_0)\quad\textrm{ and } \mathcal G_0\textrm{ is introduced in \eqref{eq:GL-loc}}\,.$$
We apply  Lemma~\ref{lem:est-E0} and get
$$\mathcal E_1(u,\Ab)\leq {(1+C\ell^\alpha)}\,  \frac1{b B_0(x_0)}\mathcal E_{bB_0(x_0),R}\big(\widetilde \chi_\ell\,u_R\big) +C\kappa\ell\Big(\kappa^3\ell^{3+\alpha} +\kappa\ell^{1+\alpha}\Big)\,,
$$
where 
$$\widetilde\chi_\ell(\tau)=\chi\left(\frac{\tau}{\ell\sqrt{\kappa H}}\right)\,, \, b=H/\kappa\,, \mbox{ and }R=\sqrt{h_{\rm ex}}\,\tilde\ell\,, $$
in conformity with  \eqref{eq:R}.\\
Note that 
${\rm supp}(1-\widetilde\chi_\ell^2)\subset [\ell\sqrt{\kappa H}/2 \,,\,+\infty)$
and ${\rm supp\,}\widetilde\chi_\ell'\subset[ \ell\sqrt{\kappa H}/2,\ell\sqrt{\kappa H}]$. Using the decay of $u_R$ established in Theorem~\ref{thm-hc2-Pa02}, we get
$$
\mathcal E_{bB_0(x_0),R}\big(\widetilde \chi_\ell\,u_R\big)\leq \mathcal E_{bB_0(x_0),R}\big(u_R\big)+C\frac{|\ln(\ell\sqrt{\kappa H})|^2}{\ell\sqrt{\kappa H}}\,.$$
Since $\mathcal E_{bB_0(x_0),R}(u_R)= d (bB_0(x_0),R)$ and $ R=(1+\sigma)\ell \sqrt{B_0(x_0)\kappa H}$, Theorem~\ref{thm-hc2-Pa02} yields
\begin{multline}\label{eq:conc-step2}
\mathcal E_1(u,\Ab)\leq 2\kappa\ell\sqrt{\frac{1}{b|B_0(x_0)|}}\,\Es\big(bB_0(x_0)\big)\\ +C\kappa\ell\Big(\ell^\alpha+\kappa^3\ell^{3+\alpha} +\kappa\ell^{1+\alpha}+\sigma+(\kappa\ell)^{-1}+\frac{|\ln(\ell\sqrt{\kappa H})|^2}{\ell^2\kappa^2}\Big)\,.
\end{multline}

{\bf Step~3: Estimating $\mathcal E_2(u,\Ab)$.}

Let $V_{x_0}(\tilde \ell )^{\,\complement}:= \Omega \setminus V_{x_0}(\tilde \ell ) $ and  $u=\eta_\ell\,\psi\,$. By the Cauchy-Schwarz inequality,  we get, for any $\zeta \in (0,1)\,,$
\begin{align*}
\int_{V_{x_0}(\tilde \ell )^{\,\complement}}&|(\nabla-i\kappa H\Ab)\eta_\ell\psi|^2\,dx\\
&\leq
(1+ \kappa^{-\zeta})\int_{V_{x_0}(\tilde \ell )^{\,\complement}}|\eta_\ell(\nabla-i\kappa H\Ab)\psi|^2\,dx+{(1+ \kappa^\zeta)}\int_{V_{x_0}(\tilde \ell )^{\,\complement}}
|\nabla\eta_\ell|^2|\psi|^2\,dx\\
&\leq 
(1+ \kappa^{-\zeta})\int_{V_{x_0}(\tilde \ell )^{\,\complement}}|(\nabla-i\kappa H\Ab)\psi|^2\,dx\\
&\qquad
+{(1+\kappa^\zeta)}\int_{\{t(x)\leq \sigma\ell\}\cap V_{x_0}(\tilde \ell )^\complement}
|\nabla\eta_\ell|^2|\psi|^2\,dx
+ {(1+\kappa^\zeta)}\int_{\{t(x)> \sigma\ell\}}|\nabla\eta_\ell|^2|\psi|^2\,dx\\
&\leq(1+{\kappa^{-\zeta}})\int_{V_{x_0}(\tilde \ell )^{\,\complement}}|(\nabla-i\kappa H\Ab)\psi|^2\,dx\, +\,  { C (1+\kappa^\zeta)}\,.
\end{align*}
Here we used the properties of the function $\eta_\ell$, namely that $\eta_\ell\leq 1$, $|\nabla\eta_\ell|=\mathcal O(\sigma^{-1}\ell^{-1})$ and $|\{t(x)\leq \sigma\ell\}\cap V_{x_0}(\tilde \ell )^\complement|=\mathcal O(\sigma^2\ell^2)\,$. \\
For the integral over $\{t(x)>\sigma\ell\}$, we use that $b|B_0(x_0)|\geq 1+\epsilon\,$, which in turn allows us to use Theorem~\ref{thm:exp-dec*} and prove that the integral of $|\psi|^2$ is exponentially small as $ \kappa \rightarrow +\infty\,$.

Now we use that $\Big|V_{x_0}(\tilde \ell )^{\,\complement}\cap {\rm supp}(1-\eta_\ell)\Big|=\mathcal O(\sigma\ell^2)$ 
to write
\begin{align*}
-\kappa^2\int_{V_{x_0}(\tilde \ell )^{\,\complement}}|\eta_\ell\psi|^2\,dx&=
-\kappa^2\int_{V_{x_0}(\tilde \ell )^{\,\complement}}|\psi|^2\,dx+\kappa^2\int_{V_{x_0}(\tilde \ell )^{\,\complement}}(1-\eta_\ell^2)|\psi|^2\,dx\\
&\leq -\kappa^2\int_{V_{x_0}(\tilde \ell )^{\,\complement}}|\psi|^2\,dx+C\kappa^2\sigma\ell^2\,.
\end{align*}
This yields
\begin{multline*}
\mathcal E_2(u,\Ab)\leq {(1+\kappa^{-\zeta})} \int_{\Omega\setminus V_{x_0}(\tilde \ell )}\Big(|(\nabla-i\kappa H\Ab)\psi|^2-\kappa^2|\psi|^2+\frac{\kappa^2}2|\psi|^4\Big)\,dx\\+C\Big(1+\kappa^\zeta+\kappa^2\sigma\ell^2\Big)
+\kappa^2H^2\int_\Omega|\curl\Ab-B_0|^2\,dx\,.
\end{multline*}
Remembering the definition of $\mathcal E_2(\psi,\Ab)$ in \eqref{eq:E1,E2}, we obtain
\begin{equation}\label{eq:E2(u,A)}
\mathcal E_2(u,\Ab)\leq (1+\kappa^{-\zeta}) \, \mathcal E_2 (\psi,\Ab) +C\, \Big(1+\kappa^{\zeta}+\kappa^2\sigma\ell^2\Big)\,.
\end{equation}

{\bf Step~4: Upper bound of the local Ginzburg-Landau energy.}\\

\noindent Since $(\psi,\Ab)$ is a minimizer of the functional $\mathcal E(\cdot,\cdot)$, $\mathcal E(\psi,\Ab)\leq \mathcal E(0,\Ab_0)=0$ and 
$$\mathcal E(\psi,\Ab)\leq \mathcal E(u,\Ab)=\mathcal E_1(u,\Ab)+\mathcal E_2(u,\Ab)\,.$$
Using that $\mathcal E(\psi,\Ab)\leq 0\,$, we get further
$$ (1+ \kappa^{-\zeta}) \mathcal E(\psi,\Ab)\leq \mathcal E(u,\Ab)=\mathcal E_1(u,\Ab)+\mathcal E_2(u,\Ab)\,.$$
 By \eqref{eq:E1,E2}, we may write the simple identity $\mathcal E(\psi,\Ab)=\mathcal E_1(\psi,\Ab)+\mathcal E_2(\psi,\Ab)$. Using \eqref{eq:E2(u,A)}, we get
$$(1+\kappa^{-\zeta})\,\mathcal E_1(\psi,\Ab)\leq \mathcal E_1(u,\Ab)+C\Big(1+\kappa^{\zeta}+\kappa^2\sigma\ell^2\Big) 
\,.$$
Now, we use the estimate in \eqref{eq:conc-step2} to write
\begin{multline}\label{eq:conc-step3}
\mathcal E_1(\psi,\Ab)\leq 2\kappa\ell\sqrt{\frac{1}{b|B_0(x_0)|}}\,\Es\big(bB_0(x_0)\big) \\+C\kappa\ell\Big(\kappa^{-\zeta}+\ell^\alpha+\kappa^3\ell^{3+\alpha} +\kappa\ell^{1+\alpha}+\sigma+(\kappa\ell)^{-1}+\kappa^{-1+\zeta}\ell^{-1}+\kappa\sigma\ell+\frac{|\ln(\ell\sqrt{\kappa H})|^2}{\ell^2\kappa^2}\Big)
 \,.
\end{multline}

{\bf Step~5: Lower bound of the $L^4$-norm.}

We select 
$$\ell=\kappa^{-\rho}\,,\quad \sigma=\kappa^{-\delta}\quad{\rm and}\quad  \zeta=\frac{1-\rho}2\,,$$
with
$$\frac1{1+\alpha}<\frac{3}{3+\alpha}<\rho<1\quad{\rm and}\quad 1-\rho<\delta<1\,.$$
In this way, we get that, the restriction $\bar\Sigma (\kappa,\kappa^{-\rho}, \kappa^{-\delta})$  of 
\begin{equation}\label{eq:S-bar}
\bar\Sigma (\kappa,\ell,\sigma):=\kappa^{-\zeta}+\ell^\alpha+\kappa^3\ell^{3+\alpha} +\kappa\ell^{1+\alpha}+\sigma+(\kappa\ell)^{-1}+\kappa^{-1+\zeta}\ell^{-1}+\kappa\sigma\ell+\frac{|\ln(\ell\sqrt{\kappa H})|^2}{\ell^2\kappa^2}\,,
\end{equation}
 tends to $0$ as $\kappa \rightarrow +\infty\,$.\\
Consequently, we infer from \eqref{eq:conc-step3}, 
\begin{equation}\label{eq:conc-step3*}
\mathcal E_1(\psi,\Ab)\leq 2\kappa\, \ell\, \sqrt{\frac{1}{b|B_0(x_0)|}}\,\Es\big(bB_0(x_0)\big) +C\, \kappa\ell \, \bar\Sigma(\kappa,\kappa^{-\rho},\kappa^{-\delta}) \,.
\end{equation}
Now, let $f$ be the smooth function satisfying \eqref{eq:f*}. Again, using the properties of $f$ and a straightforward computation as in Step~3, we have
\begin{equation}\label{eq:step4}
\begin{aligned}
&\mathcal E_1(f\psi,\Ab)\leq (1+\kappa^{-\zeta})\, \mathcal E_1(\psi,\Ab)+C\Big(\kappa^{\zeta}+\kappa^2\sigma\ell^2\Big)\,,\\
&\int_{V_{x_0}(\tilde \ell )}f^2\Big(-1+\frac12f^2\Big)|\psi|^4\,dx\geq -\frac12\int_{V_{x_0}(\ell)}|\psi|^4\,dx
+C\sigma\ell^2\,.
\end{aligned}
\end{equation}
 Using the lower bound of $\mathcal E_1(f\psi;\Ab)$ in \eqref{eq:lb-loc-en}, we get from \eqref{eq:conc-step3*} and \eqref{eq:step4},
$$\left|\mathcal E_1(\psi,\Ab)-2\kappa \ell\, \sqrt{\frac{1}{b|B_0(x_0)|}}\,\Es\big(bB_0(x_0)\big)\right|\leq 
C\kappa\ell\,\bar\Sigma (\kappa,\kappa^{-\rho},\kappa^{-\delta})\,.$$
Remembering the definition of $\mathcal E_1(\psi,\Ab)=\mathcal E_1\big(\psi,\Ab;V_{x_0}((1+\sigma)\ell)\big)$, we get the statement concerning the local energy in Theorem~\ref{thm:lb-psi}.
 
Now we return back to \eqref{eq:decomp*}. Using \eqref{eq:step4}, we write
$$(1+\kappa^{-\zeta})\, \mathcal E_1(\psi,\Ab)+C\Big(\kappa^{-\zeta}+\kappa^2\sigma\ell^2\Big)\geq 
-\frac{\kappa^2}2\int_{V_{x_0}(\ell)}|\psi|^4\,dx
-C\sigma\ell^2\kappa^2\,.
$$
Rearranging the terms,  then using \eqref{eq:conc-step3*} and \eqref{eq:S-bar}, we arrive at the following upper bound
$$\frac{\kappa^2}2\int_{V_{x_0}(\ell)}|\psi(x)|^4\,dx\geq  -2\kappa\ell \, \sqrt{\frac{1}{b|B_0(x_0)|}}\,\Es\big(bB_0(x_0)\big) +C\,\kappa\,\ell\, \bar\Sigma (\kappa,\kappa^{-\rho},\kappa^{-\delta})\,.$$
 Using the remark around \eqref{eq:S-bar},  this finishes the proof of Theorem~\ref{thm:lb-psi}.
\end{proof}

\section{The superconductivity region: Proof of Theorem~\ref{thm:blk-sc*}}\label{sec:blk}

In this section, we present  the proof of Theorem~\ref{thm:blk-sc*} devoted to  the distribution of the superconductivity in the region 
$$\{x\in\overline\Omega,~b\,|B_0(x)|<1\}\quad \textrm{for the applied magnetic field }H=b\kappa\,.$$
The proof follows by an analysis similar to the one in Section~\ref{sec:surf}, so our presentation  will be shorter  here.

\begin{rem}\label{rem:blk-sc}
As $\ell\to0_+$, the area of  $\mathcal W(x_0,\ell)$ as introduced in \eqref{defW}  is
$$|\mathcal W(x_0,\ell)|=4\ell^2\quad{\rm if~}x_0\in\Omega\,,$$
and
$$|\mathcal W(x_0,\ell)|=4\ell^2+o(\ell^2)\quad{\rm if~}x_0\in\partial\Omega\,.$$
\end{rem}


The proof of Theorem~\ref{thm:blk-sc*} is presented in five steps. In the sequel, $\rho\in(\frac{2}{2+\alpha},1)$ and $c_2>c_1>0$ are fixed,
\begin{equation}\label{eq:ell-sigma}
c_1\kappa^{-\rho}\leq \ell\leq c_2\kappa^{-\rho}\quad{\rm and}\quad \sigma=\kappa^{\frac{\rho-1}2}\,.
\end{equation}
We will refer to the condition on $\ell$ by writing $\ell\approx\kappa^{-\rho}$.

~

\subsection*{Step~1. Useful estimates.}~

Let $f$ be a smooth function such that
\begin{equation}\label{eq:f**}
f=1\quad{\rm in~} \mathcal W(x_0,\ell),\quad 0\leq f\leq 1{~\rm and~}|\nabla f|\leq \frac{C}{\sigma\ell}\quad {\rm in~}\mathcal W\big(x_0,(1+\sigma)\ell\big)\,.
\end{equation}
As in the proof of \eqref{eq:decomp**}, we have
\begin{equation}\label{eq:decomp**blk}
\mathcal E_1\Big(f\psi,\Ab;\mathcal W(x_0,(1+\sigma)\ell)\Big)\leq -\frac{\kappa^2}2\int_{\mathcal W(x_0,\ell)}|\psi(x)|^4\,dx+C\sigma^{-1}\,.
\end{equation}
Here $\mathcal E_1$ is introduced in \eqref{eq:GLV}. Furthermore, we have the following two estimates (cf. \eqref{eq:step4}):
\begin{multline}\label{eq:lb-loc-en-blk}
\mathcal E_1\Big(f\psi,\Ab;\mathcal W(x_0,(1+\sigma)\ell)\Big)\leq
(1+\kappa^{-\zeta})\mathcal E_1\Big(\psi,\Ab;\mathcal W(x_0,(1+\sigma)\ell)\Big)\\+C\kappa^2\ell^2\big(\sigma^{-1}\kappa^{\zeta}(\kappa\ell)^{-2}+\sigma\big)\,,\end{multline}
and (cf. \eqref{eq:decomp*})
\begin{equation}\label{eq:lb-loc-en-blk*}
\begin{aligned}
\mathcal E_1\Big(f\psi,\Ab;\mathcal W(x_0,(1+\sigma)\ell)\Big)&\geq 
\kappa^2\int_{\mathcal W(x_0,(1+\sigma)\ell)}f^2\left(-1+\frac12f^2\right)|\psi|^4\,dx\\
&\geq-\frac{\kappa^2}2\int_{\mathcal W(x_0,\ell)}|\psi(x)|^4\,dx+C\sigma\ell^2\kappa^2\,,
\end{aligned}
\end{equation}
where $\zeta\in(0,1)$ is a constant to be chosen later.

~

\subsection*{Step~2. The case $B_0(x_0)=0$\,.}~

The upper bound for the integral of $|\psi|^4$ in Theorem~\ref{thm:blk-sc*} is trivial since $|\psi|\leq 1$ and $g(0)=-\frac12\,$.

 We have the obvious inequalities
$$\mathcal E_1\Big(f\psi,\Ab;\mathcal W(x_0,(1+\sigma)\ell)\Big)\geq \int_{\mathcal W(x_0,(1+\sigma)\ell)}\Big(-\kappa^2|f\psi|^2+\frac{\kappa^2}2|f\psi|^4\Big)\,dx\geq -\frac{\kappa^2}2 \int_{\mathcal W(x_0,(1+\sigma)\ell)}\,dx\,.$$
Inserting this into \eqref{eq:lb-loc-en-blk} and selecting $\zeta=\frac{1-\rho}2\,$, we get 
$$\mathcal E_1\Big(\psi,\Ab;\mathcal W(x_0,(1+\sigma)\ell)\Big)\geq -C\kappa^2\ell^2\big(\sigma^{-1}\kappa^{\zeta}(\kappa\ell)^{-2}+\sigma\big)=o(\kappa^2\ell^2)\,,$$
since $\sigma=\kappa^{\frac{\rho-1}2}$, $\ell\approx\kappa^{-\rho}$ and $\frac{2}{2+\alpha}<\rho<1\,$. 

Now we prove an upper bound for $\mathcal E_1\Big(f\psi,\Ab;\mathcal W(x_0,(1+\sigma)\ell)\Big)$.  Let $\eta_\ell$ be a smooth function satisfying
\begin{equation}\label{eq:eta-ell*}
\eta_\ell(x)=0~{\rm in~}\mathcal W(x_0,(1+\sigma)\ell)\,,\quad \eta_\ell(x)=1~{\rm in~}\Omega\setminus \mathcal W(x_0,(1+2\sigma)\ell)\,,\quad 0\leq \eta_\ell(x)\leq 1~{\rm in~}\Omega\,,
\end{equation}
and
\begin{equation}\label{eq:eta-ell**}
|\nabla\eta_\ell(x)|\leq C\sigma^{-1}\ell^{-1}\quad{\rm in~}\Omega\,,
\end{equation}
for some constant $C>0$\,.
We define the function
$$u(x)=\exp\big(i\kappa H w(x)\big)f(x)
+\eta_\ell(x)\psi(x)\,,$$
where the  function  $w$ is the sum of two functions $w_1$ and $w_2$ such that the two inequalities in \eqref{eq:w} are satisfied  in $\mathcal W(x_0,(1+\sigma)\ell))$.

We have the obvious decomposition
$$\mathcal E(u,\Ab)=\mathcal E_1\Big(\exp\big(i\kappa H w(x)\big)f(x),\Ab;\mathcal W(x_0,(1+\sigma)\ell)\Big)+
\mathcal E_2\Big(\eta_\ell(x)\psi(x),\Ab;\Omega\setminus\mathcal W(x_0,(1+\sigma)\ell)\Big)\,,$$
where $\mathcal E_1$ and $\mathcal E_2$ are introduced in \eqref{eq:GLV}.

We estimate $\mathcal E_2\Big(\eta_\ell(x)\psi(x),\Ab;\Omega\setminus\mathcal W(x_0,(1+\sigma)\ell)\Big)$ as we did in the proof of Theorem~\ref{thm:lb-psi} (cf. Step~3 and \eqref{eq:E2(u,A)}). In this way we get
\begin{multline}\label{eq:E2-blk*}
\mathcal E_2\Big(\eta_\ell(x)\psi(x),\Ab;\Omega\setminus\mathcal W(x_0,(1+\sigma)\ell)\Big)\leq (1+\kappa^{-\zeta})\mathcal E_2\Big(\psi(x),\Ab;\Omega\setminus\mathcal W(x_0,(1+\sigma)\ell)\Big)\\+C(\sigma^{-1}\kappa^\zeta+\sigma\kappa^2\ell^2)\,.
\end{multline}
For the term $\mathcal E_1\Big(\exp\big(i\kappa H w(x)\big)f(x),\Ab;\mathcal W(x_0,(1+\sigma)\ell)\Big)$, we argue as in the proof of Theorem~\ref{thm:lb-psi} (Step~2) and write
\begin{multline*}
\mathcal E_1\Big(\exp\big(i\kappa H w(x)\big)f(x),\Ab;\mathcal W(x_0,(1+\sigma)\ell)\Big)\\\leq (1+\ell^\alpha)\mathcal E_1\Big(f(x),B_0(x_0)\Ab_0;\mathcal W(x_0,(1+\sigma)\ell)\Big)+C(\kappa^2\ell^{2+\alpha}+\kappa^4\ell^{4+\alpha})\,.
\end{multline*}
Note that
$$
\begin{aligned}
\mathcal E_1\Big(f(x),B_0(x_0)\Ab_0;\mathcal W(x_0,(1+\sigma)\ell)\Big)&=\mathcal E_1\Big(f(x),0;\mathcal W(x_0,(1+\sigma)\ell)\Big)\\
&\leq C\sigma^{-1}+\kappa^2\int_{\mathcal W(x_0,(1+\sigma)\ell)}f^2\Big(-1+\frac{f^2}2\Big)\,dx\\
&\leq C\sigma^{-1}-\frac{\kappa^2}2|\mathcal W(x_0,(1+\sigma)\ell)|+C\sigma\kappa^2\ell^2\,.
\end{aligned}
$$
Therefore, we get the estimate
\begin{multline*}
\mathcal E_1\Big(\exp\big(i\kappa H w(x)\big)f(x),\Ab;\mathcal W(x_0,(1+\sigma)\ell)\Big)\leq -(1+\ell^\alpha)\frac{\kappa^2}2|\mathcal W(x_0,(1+\sigma)\ell)|\\
+C\kappa^2\ell^2(\ell^{\alpha}+\kappa^2\ell^{2+\alpha}+\sigma^{-1}(\kappa\ell)^{-2}+\sigma)\,,\end{multline*}
and consequently
\begin{multline*}
\mathcal E(u,\Ab)\leq -\frac{\kappa^2}2|\mathcal W(x_0,(1+\sigma)\ell)|+
 (1+\kappa^{-\zeta})\mathcal E_2\Big(\psi(x),\Ab;\Omega\setminus\mathcal W(x_0,(1+\sigma)\ell)\Big)\\
 +C\kappa^2\ell^2(\ell^{\alpha}+\kappa^2\ell^{2+\alpha}+\sigma^{-1}\kappa^\zeta(\kappa\ell)^{-2}+\sigma)\,.\end{multline*}
Using that $\mathcal E(\psi,\Ab)\leq \min\Big(0,\mathcal E(\psi,\Ab)\Big)$, we get
\begin{multline}\label{eq:E1-blk*}
(1+\kappa^{-\zeta})\mathcal E_1\Big(\psi,\Ab;\mathcal W(x_0,(1+\sigma)\ell)\Big)\leq -\frac{\kappa^2}2|\mathcal W(x_0,(1+\sigma)\ell)|\\
 +C\kappa^2\ell^2(\ell^{\alpha}+\kappa^2\ell^{2+\alpha}+\sigma^{-1}\kappa^\zeta(\kappa\ell)^{-2}+\sigma)\,.\end{multline}
We insert this into \eqref{eq:lb-loc-en-blk}, then we substitute  the resulting inequality into \eqref{eq:lb-loc-en-blk*}. In this way we get
$$
\int_{\mathcal W(x_0,\ell)}|\psi|^4\,dx\geq \frac12|\mathcal W(x_0,(1+\sigma)\ell)|-C(\sigma+\sigma^{-1}\kappa^\zeta(\kappa\ell)^{-2}+\kappa^2\ell^{2+\alpha}+\ell^\alpha+\kappa^{-\zeta})\,.
$$
The assumption on $\sigma$ and $\ell$ in \eqref{eq:ell-sigma} and the choice $\zeta=\frac{1-\rho}2$ yield that the term on the right hand side above is $o(1)$, hence  we get the lower bound for the integral of $|\psi|^4$  in Theorem~\ref{thm:blk-sc*}. Now,  the estimate of the energy follows by collecting the estimates in \eqref{eq:E1-blk*} and \eqref{eq:lb-loc-en-blk*}.

~

\subsection*{Step~3.  The case $|B_0(x_0)|>0$: Upper bound.}~

 We use \eqref{eq:ke-1} and \eqref{eq:ke-2} with $\gamma=\delta=\alpha$. We obtain, for some $C^{2,\alpha}$ real-valued function $w$,
\begin{multline}\label{eq:E1-lb-blk}
\mathcal E_1\Big(f\psi,\Ab;\mathcal W(x_0,(1+\sigma)\ell)\Big)\geq
\mathcal (1-\ell^\alpha)\mathcal E_1\Big(e^{-i\kappa H w}f\psi,\Ab_0;\mathcal W(x_0,(1+\sigma)\ell)\Big)\\-C\kappa^2\ell^2(\ell^\alpha+\kappa^2\ell^{2+\alpha})\,.
\end{multline}
If $x_0\in\Omega\,$, we get by re-scaling  and \eqref{eq:g(b)*} that 
$$\mathcal E_1\Big(e^{-i\kappa Hw}f\psi,\Ab_0;\mathcal W(x_0,(1+\sigma)\ell)\Big)\geq 4\kappa^2(1+\sigma)^2\ell^2 g(b|B_0(x_0)|)\,.$$
If $x_0\in\partial\Omega\,$, then we may write a lower bound for $\mathcal E_1\Big(f\psi,\Ab_0;\mathcal W(x_0,(1+\sigma)\ell)\Big)$ by converting to boundary coordinates as in Lemma~\ref{lem:est-E0} and get
\begin{align*}
\mathcal E_1\Big(e^{-i\kappa Hw}f\psi,\Ab_0;&\mathcal W(x_0,(1+\sigma)\ell)\Big)\\
&\geq \frac{(1-C\ell)}{b|B_0(x_0)|} e^N\Big(b|B_0(x_0)|,2(1+\sigma)\ell
\sqrt{|B_0(x_0)|\kappa H}\Big)-C\kappa^2\ell^2
\big(\ell+\kappa^2\ell^3\big)\\
&\geq 4\kappa^2(1+\sigma)^2\ell^2 g(b|B_0(x_0)|)-C\kappa^2\ell^2\big(\ell+\kappa^2\ell^3+(\kappa\ell)^{-1}\big)\,.
\end{align*}
Thus, we infer from \eqref{eq:E1-lb-blk}, for $x_0\in\overline\Omega\,$,
$$\mathcal E_1\Big(f\psi,\Ab;\mathcal W(x_0,(1+\sigma)\ell)\Big)\geq 4\kappa^2(1+\sigma)^2\ell^2 g(b|B_0(x_0)|)-C\kappa^2\ell^2\big(\ell^\alpha+\kappa^2\ell^{2+\alpha}+(\kappa\ell)^{-1}\big)\,.$$
Inserting this into \eqref{eq:decomp**blk}, we get
$$\frac12\int_{\mathcal W(x_0,\ell)}|\psi (x) |^4\,dx\leq 4(1+\sigma)^2\ell^2 g(b|B_0(x_0)|)+C\ell^2\big(\ell^\alpha+\kappa^2\ell^{2+\alpha}+(\kappa\ell)^{-1}+(\kappa\ell)^{-2}\sigma^{-1}\big)\,.$$
Our choice of $\sigma$ and $\ell$ in \eqref{eq:ell-sigma} guarantees  that the term on the right side above is $o(\ell^2)\,$. Using Remark~\ref{rem:blk-sc}, we get the upper bound in Theorem~\ref{thm:blk-sc*}\,.

\begin{rem}\label{rem:step3}
The proof in step~3 is still valid if $|B_0(x_0)|\geq \kappa^{-2\gamma}$, $0<\gamma<1-\rho$  and $  Q_{4\kappa^{-\rho}}(x_0)\subset\Omega\,.$
\end{rem}
~ 
\subsection*{Step~4. The case $|B_0(x_0)|>0$ and $x_0\in\partial\Omega\,$: Lower bound.} ~

For the sake of simplicity, we treat the case $B_0(x_0)>0$. The case $B_0(x_0)<0$ can be treated similarly by taking complex conjugation.

We define the function
$$u(x)=\mathbf 1_{\mathcal W(x_0,(1+\sigma)\ell)}(x)\exp\big(i\kappa H w(x)\big)w_R\circ\Phi_{x_0}^{-1}(x)
+\eta_\ell(x)\psi(x)\,,$$
where the function $\eta_\ell$ satisfies \eqref{eq:eta-ell*} and \eqref{eq:eta-ell**}. Similarly as in \eqref{eq:u}, the function  $w$ is the sum of two functions $w_1$ and $w_2$, defined in $\mathcal W(x_0,(1+\sigma)\ell))$ and satisfying  the two inequalities in \eqref{eq:w}. Finally
$$w_R(s,t)=
\exp\big(i\kappa H g_{x_0}(s,t)\big)\exp\left(\frac{-i\kappa H st}{2}\right)u_R\big(s\sqrt{B_0(x_0)\kappa H}\,,\,t\sqrt{B_0(x_0)\kappa H}\,\big)\,,
$$
and $g_{x_0}$ is the function satisfying  \eqref{eq:vR} in $\mathcal W(x_0,\ell)$ (by Lemma~\ref{eq-hc2-gaugeT}). The function $u_R\in H^1_0(Q_{R})$ is a minimizer of the  energy 
$e^D\big(bB_0(x_0),R\big)$ for $R=2(1+\sigma)\sqrt{B_0(x_0)\kappa H}$ (cf. \eqref{eq:eD}). We can estimate $\mathcal E(u,\Ab)$ similarly as we did in the proof of Theorem~\ref{thm:lb-psi} and get
\begin{multline*}
\mathcal E(u,\Ab)\leq 4(1+\sigma)^2\ell^2\kappa^2 g\big(bB_0(x_0)\big)+(1+\kappa^{-\zeta})\mathcal E_2(\psi,\Ab)\\
+C\kappa^2\ell^2\big(\ell^\alpha+\kappa^2\ell^{2+\alpha}+\sigma +\sigma^{-1}\kappa^{\zeta}(\kappa\ell)^{-2}\big)\,,
\end{multline*}
where $\zeta\in(0,1)$ will be chosen later and
\begin{multline*}
\mathcal E_2(\psi,\Ab)=\int_{\Omega\setminus \mathcal W(x_0,(1+\sigma)\ell)}\Big(|(\nabla-i\kappa H\Ab)\psi|^2-\kappa^2|\psi|^2+\frac{\kappa^2}2|\psi|^4  \Big)\,dx\\
+\kappa^2H^2\int_\Omega|\curl\Ab-B_0|^2\,dx\,.\end{multline*}
Now we use that $\mathcal E(\psi,\Ab)\leq \min(0,\mathcal E(u,\Ab))$ to write
\begin{multline}\label{eq:ub-loc-en-blk}
(1+\kappa^{-\zeta})\mathcal E_1\Big(\psi,\Ab;\mathcal W(x_0,(1+\sigma)\ell)\Big)\leq 4(1+\sigma)^2\ell^2\kappa^2 g\big(bB_0(x_0)\big)\\
+C\kappa^2\ell^2\big(\ell^\alpha+\kappa^2\ell^{2+\alpha}+\sigma +\sigma^{-1}\kappa^{\zeta}(\kappa\ell)^{-2}\big)\,.
\end{multline}
Now we use \eqref{eq:lb-loc-en-blk} and \eqref{eq:lb-loc-en-blk*} to obtain
\begin{multline*}
-\frac{\kappa^2}2\int_{\mathcal W(x_0,\ell)}|\psi(x)|^4\,dx+C\sigma\ell^2\kappa^2\leq 4(1+\kappa^{-\zeta})(1+\sigma)^2\ell^2\kappa^2 g\big(bB_0(x_0)\big)\\
+C\kappa^2\ell^2\big(\ell^\alpha+\kappa^2\ell^{2+\alpha}+\sigma +\sigma^{-1}\kappa^{-\zeta}(\kappa\ell)^{-2}\big)\,.
\end{multline*}
We select $\zeta=\frac{1-\rho}2$. Remembering that $\sigma=\kappa^{\frac{\rho-1}2}$ and $\ell\approx\kappa^{-\rho}$ (cf. \eqref{eq:ell-sigma}), we get the lower bound for the integral of $|\psi|^4$ as in Theorem~\ref{thm:blk-sc*}.

For the estimate of the local energy $\mathcal E_1(\psi,\Ab;\mathcal W(x_0,(1+\sigma)\ell))$, we collect the inequalities 
in \eqref{eq:ub-loc-en-blk}, \eqref{eq:lb-loc-en-blk}, \eqref{eq:lb-loc-en-blk*} and the lower and upper bounds for the integral of $|\psi|^4$.

\begin{rem}\label{rem:step4}
Remark~\ref{rem:step3} holds for Step~4 as well.
\end{rem}
~

\subsection*{Step 5. The case $|B_0(x_0)|>0$ and $x_0\in\Omega\,$: Lower bound.}~

In this case $\mathcal W_{x_0}((1+\sigma)\ell)=Q_{2(1+\sigma)\ell}(x_0)$. We define the following trial state
$$u(x)=\mathbf 1_{\mathcal W(x_0,(1+\sigma)\ell)}(x)\exp\big(i\kappa H w(x)\big)w_R(x)
+\eta_\ell(x)\psi(x)\,,$$
where the functions $w$ and $\eta_\ell$ are as in Step~4, 
$$w_R(s,t)=\left\{
\begin{array}{l}
u_R\big(\sqrt{B_0(x_0)\kappa H}\,(x-x_0)\big)~{\rm if}~B_0(x_0)>0\,,\\
~\\
\overline{u_R\big(\sqrt{B_0(x_0)\kappa H}\,(x-x_0)\big)}~{\rm if}~B_0(x_0)<0\,,
\end{array}\right.
$$
and $u_R\in H^1_0(Q_R)$ is a minimizer of the  energy 
$e^D\big(bB_0(x_0),R\big)$ for $R=2(1+\sigma)\sqrt{B_0(x_0)\kappa H}$ (cf. \eqref{eq:eD}).

We argue as in Step~4 and obtain the lower bound for the integral of $|\psi|^4$ in Theorem~\ref{thm:blk-sc*}. The details are omitted.

\subsection*{Acknowledgements}
%
A.K.
 acknowledges  financial support  from the Lebanese
University through {\it `Equipe de Modelisation, Analyse et Applications'}. 
This work  was achieved when the first author visited 
the Center for Advanced Mathematical Sciences at the American University of Beirut in the framework of the {\it Atiyah Distinguished Visitors Program.
}


\begin{thebibliography}{100}
\bibitem{AS} A. Aftalion, S. Serfaty. 
\newblock Lowest Landau level approach for the Abrikosov lattice close to the second critical field. 
\newblock Selecta Math. 2,13, (2007).



\bibitem{Al} Y. Almog. 
\newblock Non-linear surface superconductivity in three dimensions in the large $\kappa$ limit.
\newblock {\it Commun. Contemp. Math.} {\bf 6} (4), 637-652 (2004). 

\bibitem{AH} Y. Almog, B. Helffer. 
The distribution of surface superconductivity
along the boundary: on a conjecture of X. B. Pan.
\newblock{\it SIAM J. Math. Anal.}
{\bf 38}, 1715--1732 (2007).

\bibitem{AlHe} Y. Almog, B. Helffer. 
\newblock Global stability of the normal state
  of superconductors in the presence of a strong electric current.
 \newblock {\it Comm. Math. Phys.} {\bf 330}, 1021--1094 (2014).

\bibitem{AHP} Y. Almog, B. Helffer, X.B. Pan. 
 Mixed normal-superconducting states in the presence of strong electric currents. {\it arXiv:1505.063227}.

\bibitem{AK} W. Assaad, A. Kachmar. The influence of magnetic steps on bulk superconductivity. 
{\it Discrete and Continuous Dynamical Systems, Series A.}  (In press).

\bibitem{Att} K. Attar.
\newblock  The ground state energy of the two
dimensional Ginzburg-Landau functional with variable magnetic field.
\newblock {\it Annales de l'Institut Henri Poincar\'e - Analyse
Non-Lin\'eaire} {\bf 32}, 325-345 (2015).

\bibitem{Att2} K. Attar. 
\newblock Energy and vorticity of the Ginzburg-Landau model with variable magnetic field.
\newblock {\it Asymptotic Analysis} {\bf 93}, 75-114 (2015).

\bibitem{Att3} K. Attar. 
\newblock Pinning with a variable magnetic field of the two dimensional Ginzburg-Landau model. 
\newblock  To appear in {\it Non-linear Analysis: Theory Methods and Applications} (2015).


\bibitem{BonF} V. Bonnaillie-No\"el, S. Fournais. Superconductivity in domains with corners. {\it  Rev.  Math. Phys.} {\bf 19} (2007), pp. 607-637.

\bibitem{CL} A. Contreras, X. Lamy. 
\newblock Persistence of superconductivity
in thin shells beyond $H_{c1}$. 
\newblock {\it arXiv:1411.1078}
(2014).

\bibitem{CR} M. Correggi, N. Rougerie. 
\newblock On the Ginzburg-Landau
functional in the surface superconductivity regime.
 \newblock {\it Commun. Math. Phys.} {\bf 332},  1297-1343 (2014).

\bibitem{DR} N. Dombrowski, N. Raymond.
\newblock Semi-classical analysis with vanishing magnetic fields.
\newblock {\it Journal of Spectral Theory.} {\bf 3}, 423-464 (2013).

\bibitem{FoHe04} S. Fournais, B. Helffer.
Energy asymptotics for type II superconductors.  {\it Calc. Var.
Partial Differential Equations} {\bf 24} (3),   341--376 (2005).

\bibitem{FH-b} S. Fournais, B. Helffer. {\it Spectral Methods in
Surface Superconductivity.} Progress in Nonlinear Differential
Equations and Their Applications {\bf 77} Birkh\"auser (2010).




\bibitem{FK-am} S. Fournais, A. Kachmar.
\newblock Nucleation of bulk superconductivity close to critical magnetic field.
\newblock {\it Adv. Math.} {\bf 226},  1213--1258 (2011).


\bibitem{FK-cpde} S. Fournais, A. Kachmar. 
\newblock The ground state energy
of the three dimensional Ginzburg-Landau functional. Part~I. Bulk
regime. 
\newblock {\it Communications in Partial Differential Equations} {\bf
38}, 339--383 (2013).


\bibitem{FHP} S. Fournais, B. Helffer, M. Persson. Superconductivity between Hc2 and Hc3. \newblock{\it J.
Spectr. Theory} {\bf 1},  273--298 (2011).


\bibitem{FKP-jmpa} S. Fournais, A. Kachmar, and M. Persson. The ground state energy
of the three dimensional Ginzburg-Landau functional. Part~II.
Surface regime. {\it J. Math. Pures Appl.} {\bf 99}, 343--374 (2013).




\bibitem{HeMo-JFA} B. Helffer, A. Mohamed. Semiclassical analysis for the ground state
energy of a Schr\"odinger operator with magnetic wells. {\it J. Funct. Anal.}
{\bf 138} (1),  40-81 (1996).


\bibitem{HM} B. Helffer, A. Morame. 
\newblock 
Magnetic bottles in superconductivity. 
\newblock {\it J. Funct. Anal.}  {\bf 185} (2), 604-680 (2001).



\bibitem{HK} B. Helffer, A. Kachmar. 
\newblock The Ginzburg-Landau functional
with a vanishing magnetic field. 
\newblock {\it Arch. Ration. Mech. Anal.} {\bf 218} (1), 55-122 (2015).

\bibitem{HK2} B. Helffer, A. Kachmar.
\newblock  From constant to non-degenerately vanishing magnetic fields in superconductivity. 
 \newblock To appear in {\it Ann. Institut Henri Poincar\'e (Section Analyse non-lin\'eaire)} (2015).

\bibitem{HKo} B. Helffer, Y.A. Kordyukov. 
\newblock Spectral gaps for periodic Schr\"odinger operators with hypersurface magnetic wells: Analysis near the bottom. 
\newblock {\it J. Funct. Anal.}  
{\bf 257} (10), 3043-3081 (2009).


\bibitem{HePa} B.~Helffer and  X-B.~Pan.
\newblock Upper critical field and location of surface nucleation of superconductivity.
\newblock {\it Ann. Institut Henri Poincar\'e (Section Analyse Non-Lin\'eaire)} {\bf 20} (1), 145-181   (2003).

\bibitem{K-CMB} A. Kachmar. A new formula for the energy of bulk superconductivity. {\it Canadian Mathematical Bulletin.} 
doi:10.4153/CMB-2016-004-x (in press).

\bibitem{KN-3D}  A. Kachmar, M. Nassrallah. The distribution of 3D superconductivity near the second critical field. {\it arXiv:1511.08565}.

\bibitem{Kac} A. Kachmar. The Ginzburg-Landau order parameter near the second critical field. {\it SIAM J. Math. Anal.}
 {\bf 46} (1), 572-587 (2014).

\bibitem{K-JFA} A. Kachmar. The ground state energy of the three-dimensional
Ginzburg-Landau model in the mixed phase. {\it J. Funct. Anal.} {\bf 261} 3328-3344 (2011).

\bibitem{LuPajmp} K. Lu, X.B. Pan.
\newblock Eigenvalue problems of Ginzburg-Landau operator in bounded domains.
\newblock {\it Journal of Mathematical Physics} {\bf 40} (6) (1999), 2647--2670.

\bibitem{LuPa1} K.~Lu and  X-B.~Pan.
\newblock Estimates of the upper critical field for the
Ginzburg-Landau equations of superconductivity.
\newblock {\it Physica D } {\bf 127}, 73-104  (1999).


\bibitem{M} R. Montgomery. Hearing the zero locus of a magnetic
field. {\it Commun. Math. Phys.} {\bf 168} (3), 651--675 (1995).

\bibitem{Miq} J-P. Miqueu.
\newblock Equation de Schr\"odinger avec un champ magn\'etique qui s'annule.
\newblock Th\`ese de doctorat (in preparation).

\bibitem{Pa02}
X.B. Pan. Surface superconductivity in applied magnetic fields above
$H_{C_2}$. {\it Commun. Math. Phys.} {\bf 228}, 228-370 (2002).



\bibitem{PK} X.B. Pan, K.H. Kwek. Schr\"odinger operators with
non-degenerately vanishing magnetic fields in bounded domains. {\it
Trans. Amer. Math. Soc.} {\bf 354} (10), 4201-4227 (2002).

\bibitem{Ra} N. Raymond. Sharp asymptotics for the Neumann Laplacian with
variable magnetic field : case of dimension 2.
\newblock  {\it Annales Henri Poincar\'e.} {\bf 10} (1), 95-122 (2009).

\bibitem{R} J. Rubinstein. Six lectures on superconductivity. {\it In: Boundaries, Interfaces, and Transitions, CRM
Proceedings and Lecture Notes,} Vol. 13, Providence, RI: Am. Math. Soc.,  163--184 (1998).

\bibitem{SJ-dG} D. St. James, P.G. de\,Gennes.  Onset of superconductivity in
decreasing fields. Physics Letters 7 (5), 306-308 (1963).

\bibitem{SaSe} E. Sandier, S. Serfaty. {\it Vortices for the Magnetic
Ginzburg-Landau Model.} Progress in Nonlinear Differential Equations
and their Applications  {\bf 70}  Birkh\"auser  (2007).

\bibitem{SS02} E. Sandier, S. Serfaty. The decrease of bulk
superconductivity close to the second critical field in the
Ginzburg-Landau model. {\it SIAM J. Math. Anal.} {\bf 34} (4), 939-956
(2003).
%

\end{thebibliography}
\end{document}